\pdfoutput=1
\RequirePackage{ifpdf}
\ifpdf 
\documentclass[pdftex]{sigma}
\else
\documentclass{sigma}
\fi

\usepackage{amscd}
\usepackage[all]{xy}

\newtheorem{thm}{Theorem}[section]
\newtheorem{prop}[thm]{Proposition}
\newtheorem{lem}[thm]{Lemma}
\newtheorem{cor}[thm]{Corollary}

\theoremstyle{definition}
\newtheorem{dfn}[thm]{Definition}
\newtheorem{rem}[thm]{Remark}

\newcommand{\C}{\mathbb{C}}
\newcommand{\R}{\mathbb{R}}
\newcommand{\Z}{\mathbb{Z}}

\newcommand{\pt}{\mathrm{pt}}

\begin{document}

\allowdisplaybreaks

\newcommand{\arXivNumber}{1509.09194}

\renewcommand{\PaperNumber}{014}

\FirstPageHeading

\ShortArticleName{Twists on the Torus Equivariant under the 2-Dimensional Crystallographic Point Groups}

\ArticleName{Twists on the Torus Equivariant under\\ the 2-Dimensional Crystallographic Point Groups}

\Author{Kiyonori GOMI}

\AuthorNameForHeading{K.~Gomi}

\Address{Department of Mathematical Sciences, Shinshu University,\\
 3--1--1 Asahi, Matsumoto, Nagano 390-8621, Japan}
\Email{\href{mailto:kgomi@math.shinshu-u.ac.jp}{kgomi@math.shinshu-u.ac.jp}}
\URLaddress{\url{http://math.shinshu-u.ac.jp/~kgomi/}}

\ArticleDates{Received February 17, 2016, in f\/inal form March 03, 2017; Published online March 08, 2017}

\Abstract{A twist is a datum playing a role of a local system for topological $K$-theory. In equivariant setting, twists are classif\/ied into four types according to how they are realized geometrically. This paper lists the possible types of twists for the torus with the actions of the point groups of all the $2$-dimensional space groups (crystallographic groups), or equivalently, the torus with the actions of all the possible f\/inite subgroups in its mapping class group. This is carried out by computing Borel's equivariant cohomology and the Leray--Serre spectral sequence. As a byproduct, the equivariant cohomology up to degree three is determined in all cases. The equivariant cohomology with certain local coef\/f\/icients is also considered in relation to the twists of the Freed--Moore $K$-theory.}


\Classification{53C08; 55N91; 20H15; 81T45}


\section{Introduction}\label{sec:introduction}

Topological $K$-theory has recently been recognized as a useful tool for a classif\/ication of topological insulators in condensed matter physics. In Kitaev's 10-fold way \cite{Ki}, the usual complex $K$-theory and also $KO$ or Atiyah's $KR$-theory are used. These classif\/ications are in some sense the most simple cases, and a recent study of topological insulators focuses on more complicated cases. Such complicated cases arise when we take the symmetry of quantum systems into account. Then equivariant $K$-theory and its twisted version naturally f\/it into the classif\/ication scheme of such systems \cite{F-M}. Actually, as will be explained in Section~\ref{sec:quantum_system_to_K}, a certain quantum system on the $d$-dimensional space $\R^d$ invariant under a \textit{space group} provides a $K$-theory class on the $d$-dimensional torus $T^d$ equivariant under the point group of the space group. If the space group is \textit{nonsymmorphic}, then the equivariant $K$-class is naturally twisted. In the case of $d = 2$, such (twisted) equivariant $K$-theories are computed for the $17$ classes of $2$-dimensional space groups, in view of the classif\/ication of topological crystalline insulators \cite{SSG2, SSG3}. An outcome of these computations of twisted equivariant $K$-theories is the discovery of topological insulators which are essentially classif\/ied by $\Z_2$ but do not require the so-called time-reversal symmetry or the particle-hole symmetry \cite{SSG1}. This type of topological insulators is new in the sense that the known topological insulators essentially classif\/ied by $\Z_2$ so far require the time-reversal symmetry or the particle-hole symmetry.

The understanding of the importance of twisted equivariant $K$-theory in condensed matter physics leads to a mathematically natural issue: determining the possible `twists' for equivariant $K$-theory. To explain this issue more concretely, let us recall that twisted $K$-theory \cite{D-K, R} is in some sense a $K$-theory with `local coef\/f\/icients'. The datum playing the role of a `local system' admits various geometric realizations. In this paper, we realize them by \textit{twists} in the sense of~\cite{FHT}. If a compact Lie group $G$ acts on a space $X$, then graded twists on $X$ are classif\/ied by the Borel equivariant cohomology $H^1_G(X; \Z_2) \times H^3_G(X; \Z)$. Similarly, ungraded twists are classif\/ied by $H^3_G(X; \Z)$, on which we focus for a moment. (Sometimes $H^0_G(X; \Z)$ may be included in the twists, but we regard it as the degree of the $K$-theory.)

By def\/inition, the Borel equivariant cohomology $H^n_G(X; \Z)$ is the usual cohomology $H^n(EG \!\times_G\!$ $X; \Z)$ of the Borel construction $EG \times_G X$, which is the quotient of $EG \times X$ by the diagonal $G$-action, where $EG$ is the total space of the universal $G$-bundle $EG \to BG$. Associated to the Borel construction is the f\/ibration $X \to EG \times_G X \to BG$, and hence the Leray--Serre spectral sequence $E_r^{p, q}$ that converges to the graded quotient of a f\/iltration
\begin{gather*}
H^n_G(X; \Z) \supset F^1H^n_G(X; \Z) \supset F^2H^n_G(X; \Z) \supset \cdots
\supset F^{n+1}H^n_G(X; \Z) = 0.
\end{gather*}
One can interpret $F^pH^3_G(X; \Z) \subset H^3_G(X; \Z)$ geometrically in the classif\/ication of twists, and there are four types (see Section~\ref{sec:spectral_seq_and_twist} for details):

\begin{itemize}\itemsep=0pt
\item[(i)]
Twists which can be represented by group $2$-cocycles of $G$ with coef\/f\/icients in the trivial $G$-module $U(1)$. These twists are classif\/ied by $F^3H^3_G(X; \Z)$.

\item[(ii)]
Twists which can be represented by group $2$-cocycles of $G$ with coef\/f\/icients in the group $C(X, U(1))$ of $U(1)$-valued functions on $X$ regarded as a (right) $G$-module by pull-back. These twists are classif\/ied by $F^2H^3_G(X; \Z)$.

\item[(iii)]
Twists which can be represented by central extensions of the groupoid $X/\!/G$. These twists are classif\/ied by $F^1H^3_G(X; \Z)$.

\item[(iv)]
Twists of general type, classif\/ied by $F^0H^3_G(X; \Z) = H^3_G(X; \Z)$.
\end{itemize}

The equivariant twists on $T^d$ arising from quantum systems on $\R^d$, to be explained in Section~\ref{sec:quantum_system_to_K}, belong to $F^2H^3_P(T^d; \Z)$ with $P$ the point group of a $d$-dimensional space group $S$, and so are the twists considered in~\cite{SSG2}. Now, the mathematical issue is whether the twists arising in this way cover all the possibilities or not. The present paper answers this question in the case of $d = 2$ by a theorem (Theorem~\ref{thm:main}).

To state the theorem, let $S$ be a $2$-dimensional space group, which is also known as a $2$-dimen\-sional crystallographic group, a plane symmetry group, a wallpaper group, and so on. It is a~subgroup of the Euclidean group $\R^2 \rtimes {\rm O}(2)$ of isometries of $\R^2$, and is an extension of a~f\/inite group $P \subset {\rm O}(2)$ called the point group by a rank $2$ lattice $\Pi \cong \Z^2$ of translations of $\R^2$:
\begin{gather*}
\begin{array}{c@{}c@{}c@{}c@{}c@{}c@{}c@{}c@{}c@{}c}
1 & \ \longrightarrow \ &
\R^2 & \ \longrightarrow \ &
\R^2 \rtimes {\rm O}(2) & \ \longrightarrow \ &
{\rm O}(2) & \ \longrightarrow \ &
1 \\
 & &
\cup & &
\cup & &
\cup & &
\\
1 & \ \longrightarrow \ &
\Pi & \ \longrightarrow \ &
S & \ \longrightarrow \ &
P & \ \longrightarrow \ &
1.
\end{array}
\end{gather*}
Being a normal subgroup of $S$, the lattice $\Pi \subset \R^2$ is preserved by the action of $P$ on $\R^2$ through the inclusion $P \subset {\rm O}(2)$ and the standard left action of ${\rm O}(2)$ on $\R^2$. This induces the left action of $P$ on the torus $T^2 = \R^2/\Pi$ that we will consider. Since $P$ is a f\/inite subgroup of ${\rm O}(2)$, it is the cyclic group $\Z_n$ of order $n$ or the dihedral group $D_n = \langle C, \sigma \,|\, C^n, \sigma^2, \sigma C \sigma C \rangle$ of degree~$n$ and order~$2n$. The classif\/ication of $2$-dimensional space groups has long been known, and there are $17$ types~\cite{H, Sch1}, which we label following~\cite{Shatt}. Notice that some space groups share the same point group action on $T^2$, and there arise $13$ distinct f\/inite group actions on the torus. These actions realize essentially all the possible f\/inite subgroups in the mapping class group of the torus~\cite{New}, which is isomorphic to ${\rm GL}(2, \Z)$ as is well known~\cite{Rol}.

\begin{thm} \label{thm:main}
Let $P$ be the point group of one of the $2$-dimensional space groups $S$, acting on $T^2 = \R^2/\Pi$ via $P \subset {\rm O}(2)$ as above. Then,
$H^3_{P}(T^2; \Z) = F^0H^3_P(T^2; \Z) = F^1H^3_P(T^2; \Z)$. This cohomology group and its subgroups $F^pH^3_P(T^2; \Z)$ are as in Fig.~{\rm \ref{fig:list}}.
\begin{figure}[htpb]
$$
\begin{array}{|c|c|c||c|c|c||c|c|}
\hline
\mbox{Space group $S$} & P & \mbox{ori} &
H^3_P(T^2; \Z) & F^2 & F^3 &
E^{1,2}_\infty & E^{2,1}_\infty \tsep{2pt}\bsep{2pt}\\
\hline
\mbox{\textsf{p1}} & 1 & + & 0 & 0 & 0 &
0 & 0 \tsep{2pt}\bsep{2pt}\\
\hline
\mbox{\textsf{p2}} & \Z_2 & + & 0 & 0 & 0 &
0 & 0 \tsep{2pt}\bsep{2pt}\\
\hline
\mbox{\textsf{p3}} & \Z_3 & + & 0 & 0 & 0 &
0 & 0 \tsep{2pt}\bsep{2pt}\\
\hline
\mbox{\textsf{p4}} & \Z_4 & + & 0 & 0 & 0 &
0 & 0 \tsep{2pt}\bsep{2pt}\\
\hline
\mbox{\textsf{p6}} & \Z_6 & + & 0 & 0 & 0 &
0 & 0 \tsep{2pt}\bsep{2pt}\\
\hline
\mbox{\textsf{pm}/\textsf{pg}} & D_1 & - & \Z_2^{\oplus 2} & \Z_2 & 0 &
\Z_2 & \Z_2 \tsep{2pt}\bsep{2pt}\\
\hline
\mbox{\textsf{cm}} & D_1 & - & \Z_2 & 0 & 0 &
\Z_2 & 0 \tsep{2pt}\bsep{2pt}\\
\hline
\mbox{\textsf{pmm}/\textsf{pmg}/\textsf{pgg}} & D_2 & - &
\Z_2^{\oplus 4} & \Z_2^{\oplus 3} & \Z_2 &
\Z_2 & \Z_2^{\oplus 2} \tsep{2pt}\bsep{2pt}\\
\hline
\mbox{\textsf{cmm}} & D_2 & - & \Z_2^{\oplus 2} & \Z_2 & \Z_2 &
\Z_2 & 0 \tsep{2pt}\bsep{2pt}\\
\hline
\mbox{\textsf{p3m1}} & D_3 & - & \Z_2 & 0 & 0 &
\Z_2 & 0 \tsep{2pt}\bsep{2pt}\\
\hline
\mbox{\textsf{p31m}} & D_3 & - & \Z_2 & 0 & 0 &
\Z_2 & 0 \tsep{2pt}\bsep{2pt}\\
\hline
\mbox{\textsf{p4m}/\textsf{p4g}} & D_4 & - &
\Z_2^{\oplus 3} & \Z_2^{\oplus 2} & \Z_2 &
\Z_2 & \Z_2 \tsep{2pt}\bsep{2pt}\\
\hline
\mbox{\textsf{p6m}} & D_6 & - & \Z_2^{\oplus 2} & \Z_2 & \Z_2 &
\Z_2 & 0 \tsep{2pt}\bsep{2pt}\\
\hline
\end{array}
$$
\caption{The list of $H^3_P(T^2; \Z)$ and its subgroups $F^p = F^pH^3_P(T^2; \Z)$ for the point group $P$ of each space $2$-dimensional space group $S$. The $E_\infty$-term of the Leray--Serre spectral sequence is related to these subgroups by $E^{p, 3-p}_\infty \cong F^p/F^{p+1}$. The column ``ori'' indicates ``$+$'' if $P$ preserves the orientation of $T^2$ and ``$-$'' if not. The same actions of point groups on $T^2$ are grouped in a row. Nonsymmorphic groups are \textsf{pg}, \textsf{pmg}, \textsf{pgg} and \textsf{p4g}.
}
\label{fig:list}
\end{figure}
\end{thm}

\begin{cor} \label{cor:main}
Under the same hypothesis as in Theorem {\rm \ref{thm:main}},
\begin{itemize}\itemsep=0pt
\item[$(a)$] All the twists can be represented by central extensions of $T^2/\!/P$. In particular, there is no non-trivial twist if $P$ preserves the orientation of~$T^2$.

\item[$(b)$] If $P$ does not preserve the orientation of $T^2$, then there are twists which can be represented by central extensions of $T^2/\!/P$ but not by group $2$-cocycles of $P$.

\item[$(c)$] The subgroup $F^2H^3_P(T^2; \Z)$ is generated by the twists represented by:
\begin{itemize}\itemsep=0pt
\item group $2$-cocycle of $P$ with values in $C(T^2, U(1))$ induced from a nonsymmorphic space group $S'$ such that the action of its point group $P' \cong P$ on $T^2$ is the same as $P$; and

\item group $2$-cocycle of $P$ with values in $U(1)$.
\end{itemize}
\end{itemize}
\end{cor}

As a result, all the twists classif\/ied by $F^2H^3_P(T^2; \Z)$ are relevant to topological insulators, whereas there actually exist other twists which cannot be realized by group cocycles. At present their roles in condensed matter theory seem to be unknown.

Theorem \ref{thm:main} follows from case by case computations of the equivariant cohomology $H^3_P(T^2; \Z)$ and the Leray--Serre spectral sequence. Roughly, there are three methods according to the nature of the point group actions: The f\/irst method is applied to the cases where the torus $T^2$ is the product of circles with $P$-actions, i.e., the cases of the $\Z_2$-actions arising from \textsf{p2} and \textsf{pm/pg}. In these cases, the equivariant cohomology is computed by means of the splitting of the Gysin exact sequence, as detailed in~\cite{G}. The second method is applied to the cases where the point group has no element of order $3$. In these cases, the torus $T^2$ admits an equivariant stable splitting. As a result, the equivariant cohomology of $T^2$ admits the corresponding splitting, and the Leray--Serre spectral sequence turns out to be trivial. Finally, the third method is applied to the remaining cases. In these cases, we take a $P$-CW decomposition of $T^2$ to compute the equivariant cohomology by using the Mayer--Vietoris exact sequence and the exact sequence for a pair, and then study the Leray--Serre spectral sequence. In principle, the third method is the most basic, and hence is applied to all the cases. However, to simplify the computations, we use other methods.

These computations contain enough information to determine the equivariant cohomology $H^n_{P}(T^2; \Z)$, ($n \le 2$) of the torus with the actions of the possible f\/inite subgroups in the mapping class group ${\rm GL}(2, \Z)$.

\begin{thm} \label{thm:equiv_coh}
Let $P$ be the point group of one of the $2$-dimensional space groups $S$, acting on $T^2 = \R^2/\Pi$ via $P \subset {\rm O}(2)$. For $n \le 3$, the $P$-equivariant cohomology $H^n_P(T^2; \Z)$ is as given in Fig.~{\rm \ref{fig:list_equiv_coh}}.
\begin{figure}[htpb]
$$
\begin{array}{|c|c|c||c|c|c|c|}
\hline
\mbox{Space group $S$} & P & \mbox{ori} &
H^0_P(T^2) & H^1_P(T^2) & H^2_P(T^2) & H^3_P(T^2) \tsep{2pt}\bsep{2pt}\\
\hline
\mbox{\textsf{p1}} & 1 & + &
\Z & \Z^{\oplus 2} & \Z & 0 \tsep{2pt}\bsep{2pt}\\
\hline
\mbox{\textsf{p2}} & \Z_2 & + &
\Z & 0 & \Z \oplus \Z_2^{\oplus 3} & 0 \tsep{2pt}\bsep{2pt}\\
\hline
\mbox{\textsf{p3}} & \Z_3 & + &
\Z & 0 & \Z \oplus \Z_3^{\oplus 2} & 0 \tsep{2pt}\bsep{2pt}\\
\hline
\mbox{\textsf{p4}} & \Z_4 & + &
\Z & 0 & \Z \oplus \Z_2 \oplus \Z_4 & 0 \tsep{2pt}\bsep{2pt}\\
\hline
\mbox{\textsf{p6}} & \Z_6 & + &
\Z & 0 & \Z \oplus \Z_6 & 0 \tsep{2pt}\bsep{2pt}\\
\hline
\mbox{\textsf{pm}/\textsf{pg}} & D_1 & - &
\Z & \Z & \Z_2^{\oplus 2} & \Z_2^{\oplus 2} \tsep{2pt}\bsep{2pt}\\
\hline
\mbox{\textsf{cm}} & D_1 & - &
\Z & \Z & \Z_2 & \Z_2 \tsep{2pt}\bsep{2pt}\\
\hline
\mbox{\textsf{pmm}/\textsf{pmg}/\textsf{pgg}} & D_2 & - &
\Z & 0 & \Z_2^{\oplus 4} & \Z_2^{\oplus 4} \tsep{2pt}\bsep{2pt}\\
\hline
\mbox{\textsf{cmm}} & D_2 & - &
\Z & 0 & \Z_2^{\oplus 3} & \Z_2^{\oplus 2} \tsep{2pt}\bsep{2pt}\\
\hline
\mbox{\textsf{p3m1}} & D_3 & - &
\Z & 0 & \Z_2 & \Z_2 \tsep{2pt}\bsep{2pt}\\
\hline
\mbox{\textsf{p31m}} & D_3 & - &
\Z & 0 & \Z_3 \oplus \Z_2 & \Z_2 \tsep{2pt}\bsep{2pt}\\
\hline
\mbox{\textsf{p4m}/\textsf{p4g}} & D_4 & - &
\Z & 0 & \Z_2^{\oplus 3} & \Z_2^{\oplus 3} \tsep{2pt}\bsep{2pt}\\
\hline
\mbox{\textsf{p6m}} & D_6 & - &
\Z & 0 & \Z_2^{\oplus 2} & \Z_2^{\oplus 2} \tsep{2pt}\bsep{2pt}\\
\hline
\end{array}
$$
\caption{The list of equivariant cohomology up to degree $3$.}\label{fig:list_equiv_coh}
\end{figure}
\end{thm}

Note that some specif\/ic cases are computed in the literature (e.g., \cite{ADG, AGPP, AP}).

So far we focused on ungraded twists. To complete the classif\/ication of $P$-equivariant twists on $T^2$, we need to compute the equivariant f\/irst cohomology with coef\/f\/icients in $\Z_2$, which provides the information on `gradings' of a twist. But, the computation is immediately completed by a simple application of the universal coef\/f\/icient theorem to Theorem \ref{thm:equiv_coh}. Notice that the equivariant cohomology $H^1_P(T^2; \Z_2)$ also admits a f\/iltration
\begin{gather*}
H^1_P\big(T^2; \Z_2\big) = F^0H^1_P\big(T^2; \Z_2\big) \supset F^1H^1_P\big(T^2; \Z_2\big) \supset F^2H^1_P\big(T^2; \Z_2\big) = 0.
\end{gather*}
Because the degree in question is $1$, the degeneration of the Leray--Serre spectral sequence gives the identif\/ication
\begin{gather*}
F^1H^1_P\big(T^2; \Z_2\big) = \operatorname{Hom}\big(P, \Z_2\big) = H^1_P\big(\pt; \Z_2\big),
\end{gather*}
which is a direct summand of $H^1_P(T^2; \Z_2)$ and is also computed immediately by using the knowledge of the equivariant cohomology of the space consisting of one point, $\pt = \{ \mbox{one point} \}$, in Section~\ref{subsec:some_generality}.

\begin{cor} \label{cor:equiv_coh_Z_2}
Let $P$ be the point group of one of the $2$-dimensional space groups $S$, acting on $T^2 = \R^2/\Pi$ via $P \subset {\rm O}(2)$. Then the $P$-equivariant cohomology $H^1_P(T^2; \Z_2)$ is as in Fig.~{\rm \ref{fig:list_equiv_coh_Z2}}.
\end{cor}

\begin{figure}[!h]
$$
\begin{array}{|c|c|c||c|c|c|}
\hline
\mbox{Space group $S$} & P & \mbox{ori} &
H^1_P(T^2; \Z_2) & F^1H^1_P(T^2; \Z_2) & E_\infty^{1, 0} \tsep{2pt}\bsep{2pt}\\
\hline
\mbox{\textsf{p1}} & 1 & + &
\Z_2^{\oplus 2} & 0 & \Z_2^{\oplus 2} \tsep{2pt}\bsep{2pt}\\
\hline
\mbox{\textsf{p2}} & \Z_2 & + &
\Z_2^{\oplus 3} & \Z_2 & \Z_2^{\oplus 2} \tsep{2pt}\bsep{2pt}\\
\hline
\mbox{\textsf{p3}} & \Z_3 & + &
0 & 0 & 0 \tsep{2pt}\bsep{2pt}\\
\hline
\mbox{\textsf{p4}} & \Z_4 & + &
\Z_2^{\oplus 2} & \Z_2 & \Z_2 \tsep{2pt}\bsep{2pt}\\
\hline
\mbox{\textsf{p6}} & \Z_6 & + &
\Z_2 & \Z_2 & 0 \tsep{2pt}\bsep{2pt}\\
\hline
\mbox{\textsf{pm}/\textsf{pg}} & D_1 & - &
\Z_2^{\oplus 3} & \Z_2 & \Z_2^{\oplus 2} \tsep{2pt}\bsep{2pt}\\
\hline
\mbox{\textsf{cm}} & D_1 & - &
\Z_2^{\oplus 2} & \Z_2 & \Z_2 \tsep{2pt}\bsep{2pt}\\
\hline
\mbox{\textsf{pmm}/\textsf{pmg}/\textsf{pgg}} & D_2 & - &
\Z_2^{\oplus 4} & \Z_2^{\oplus 2} & \Z_2^{\oplus 2} \tsep{2pt}\bsep{2pt}\\
\hline
\mbox{\textsf{cmm}} & D_2 & - &
\Z_2^{\oplus 3} & \Z_2^{\oplus 2} & \Z_2 \tsep{2pt}\bsep{2pt}\\
\hline
\mbox{\textsf{p3m1}} & D_3 & - &
\Z_2 & \Z_2 & 0 \tsep{2pt}\bsep{2pt}\\
\hline
\mbox{\textsf{p31m}} & D_3 & - &
\Z_2 & \Z_2 & 0 \tsep{2pt}\bsep{2pt}\\
\hline
\mbox{\textsf{p4m}/\textsf{p4g}} & D_4 & - &
\Z_2^{\oplus 3} & \Z_2^{\oplus 2} & \Z_2 \tsep{2pt}\bsep{2pt}\\
\hline
\mbox{\textsf{p6m}} & D_6 & - &
\Z_2^{\oplus 2} & \Z_2^{\oplus 2} & 0 \tsep{2pt}\bsep{2pt}\\
\hline
\end{array}
$$
\caption{The list of f\/irst equivariant cohomology groups with coef\/f\/icients $\Z_2$. The quotient group $H^1_P(T^2; \Z_2)/F^1H^1_P(T^2; \Z_2)$ is denoted with $E_\infty^{1, 0}$.}\label{fig:list_equiv_coh_Z2}
\end{figure}

The grading of twists classif\/ied by $F^1H^1_P(T^2; \Z_2) = \operatorname{Hom}(P, \Z_2)$ plays a role in a quantum system with symmetry (see Remark~\ref{rem:grading}). However, there are other gradings generally, and their roles in condensed matter theory is unknown.

As is mentioned, Atiyah's $KR$-theory is also applied to the classif\/ication of topological insulators. The symmetry of $KR$-theory however concerns $\Z_2$-actions only, and its use is limited to rather simple cases. To take more general symmetries into account, Freed and Moore introduced a $K$-theory which unif\/ies $KR$-theory and equivariant $K$-theory~\cite{F-M}. Their $K$-theory is def\/ined for a space $X$ with an action of a compact Lie group $G$ equipped with a homomorphism $\phi\colon G \to \Z_2$. The $K$-theory of Freed--Moore reduces to the $G$-equivariant $K$-theory if $\phi$ is trivial, and to the $KR$-theory if $G = \Z_2$ and $\phi$ non-trivial. There also exists the notion of twists for the Freed--Moore $K$-theory. A computation of the twisted Freed--Moore $K$-theory is carried out in~\cite{SSG2}, leading to the discovery of a novel $\Z_4$-phase.

The knowledge about the twists of the Freed--Moore $K$-theory has therefore potential importance to condensed matter physics as well, and the present paper provides it also in the case where $X$ is the torus $T^2$ and $G$ is the point group $P$ of a $2$-dimensional space group. Notice that the classif\/ication of the twists for the Freed--Moore $K$-theory parallels that of the twists for equivariant $K$-theory (actually a generalization). In general, the graded twists are classif\/ied by $H^1_G(X; \Z_2) \times H^3_G(X; \Z_\phi)$ and the ungraded twists by $H^3_G(X; \Z_\phi)$. Here $\Z_\phi$ denotes a local system for the Borel equivariant cohomology associated to the $G$-module $\Z_\phi$ such that its underlying group is $\Z$ and $G$ acts via $\phi\colon G \to \Z_2$. The cohomology group $H^n_G(X; \Z_\phi)$ also admits a f\/iltration
\begin{gather*}
H^n_G(X; \Z_\phi) \supset F^1H^n_G(X; \Z_\phi) \supset F^2H^n_G(X; \Z_\phi) \supset \cdots \supset F^{n+1}H^n_G(X; \Z_\phi) = 0.
\end{gather*}
The associated graded quotient is computed by the Leray--Serre spectral sequence, and the subgroups $F^pH^3_G(X; \Z_\phi) \subset H^3_G(X; \Z_\phi)$ have geometric interpretations as well (Proposition~\ref{prop:phi_twist_interpretation}).

To state our results in the `twisted' case, we introduce the following def\/inition for the point group $P$ of a $2$-dimensional space group $S$ that admits a non-trivial homomorphism $\phi\colon P \to \Z_2$.
\begin{itemize}\itemsep=0pt
\item
In the cases of \textsf{p2}, \textsf{p4} and \textsf{p6}, the point group $P$ is the cyclic group $\Z_{2m} = \langle C \,|\, C^{2m} \rangle$ of even order. We write $\phi_1\colon \Z_{2m} \to \Z_2$ for the unique non-trivial homomorphism given by $\phi_1(C) = -1$.

\item
In the other case, the point group $P$ is the dihedral group $D_n = \langle C, \sigma \,|\, C^n, \sigma^2, \sigma C \sigma C \rangle$
of degree $n$ and order $2n$, and $D_n$ is embedded into ${\rm O}(2)$ so that $C$ is a rotation of $\R^2$ and $\sigma$ is a ref\/lection. We def\/ine $\phi_0 \colon D_n \to \Z_2$ to be the composition of the inclusion $D_n \to {\rm O}(2)$ and $\det \colon {\rm O}(2) \to \Z_2$. Put dif\/ferently, $\phi_0(C) = 1$ and $\phi_0(\sigma) = -1$. This provides the unique non-trivial homomorphism $D_n \to \Z_2$ if $n$ is odd. In the case of even $n$, we def\/ine two more non-trivial homomorphisms $\phi_i \colon D_n \to \Z_2$ by
\begin{gather*}
\begin{cases}
\phi_1(C) = -1, \\
\phi_1(\sigma) = 1,
\end{cases} \qquad
\begin{cases}
\phi_2(C) = -1, \\
\phi_2(\sigma) = -1.
\end{cases}
\end{gather*}
\end{itemize}

\begin{thm} \label{thm:main_twisted}
Let $P$ be the point group of one of the $2$-dimensional space groups $S$, acting on $T^2 = \R^2/\Pi$ via $P \subset {\rm O}(2)$, and $\phi\colon P \to \Z_2$ a non-trivial homomorphism. Then, $H^3_{P}(T^2; \Z_\phi) = F^0H^3_P(T^2; \Z_\phi) = F^1H^3_P(T^2; \Z_\phi)$. This cohomology group and its subgroups $F^pH^3_P(T^2; \Z_\phi)$ are as in Fig.~{\rm \ref{fig:list_twisted}}.
\begin{figure}[htpb]
$$
\begin{array}{|c|c|c||c|c|c||c|c|}
\hline
\mbox{Space group $S$} & P & \phi &
H^3_P(T^2; \Z_\phi) & F^2 & F^3 &
E^{1,2}_\infty & E^{2,1}_\infty \tsep{2pt}\bsep{2pt}\\
\hline
\mbox{\textsf{p2}} & \Z_2 & \phi_1 &
\Z_2^{\oplus 4} & \Z_2^{\oplus 3} & \Z_2 &
\Z_2 & \Z_2^{\oplus 2} \tsep{2pt}\bsep{2pt}\\
\hline
\mbox{\textsf{p4}} & \Z_4 & \phi_1 &
\Z_2^{\oplus 2} & \Z_2 & \Z_2 &
\Z_2 & 0 \tsep{2pt}\bsep{2pt}\\
\hline
\mbox{\textsf{p6}} & \Z_6 & \phi_1 &
\Z_2^{\oplus 2} & \Z_2 & \Z_2 &
\Z_2 & 0 \tsep{2pt}\bsep{2pt}\\
\hline
\mbox{\textsf{pm}/\textsf{pg}} & D_1 & \phi_0 &
\Z_2^{\oplus 2} & \Z_2^{\oplus 2}& \Z_2 &
0 & \Z_2 \tsep{2pt}\bsep{2pt}\\
\hline
\mbox{\textsf{cm}} & D_1 & \phi_0 &
\Z_2 & \Z_2 & \Z_2 &
0 & 0 \tsep{2pt}\bsep{2pt}\\
\hline
\mbox{\textsf{pmm}/\textsf{pmg}/\textsf{pgg}} & D_2 & \phi_0 &
\Z_2^{\oplus 4} & \Z_2^{\oplus 4} & \Z_2^{\oplus 2} &
0 & \Z_2^{\oplus 2} \tsep{2pt}\bsep{2pt}\\
\hline
\mbox{\textsf{pmm}/\textsf{pmg}/\textsf{pgg}} & D_2 & \phi_1, \phi_2 &
\Z_2^{\oplus 6} & \Z_2^{\oplus 5} & \Z_2^{\oplus 2} &
\Z_2 & \Z_2^{\oplus 3} \tsep{2pt}\bsep{2pt}\\
\hline
\mbox{\textsf{cmm}} & D_2 & \phi_0 &
\Z_2^{\oplus 2} & \Z_2^{\oplus 2} & \Z_2^{\oplus 2} &
0 & 0 \tsep{2pt}\bsep{2pt}\\
\hline
\mbox{\textsf{cmm}} & D_2 & \phi_1, \phi_2 &
\Z_2^{\oplus 4} & \Z_2^{\oplus 3} & \Z_2^{\oplus 2} &
\Z_2 & \Z_2 \tsep{2pt}\bsep{2pt}\\
\hline
\mbox{\textsf{p3m1}} & D_3 & \phi_0 &
\Z_2 & \Z_2 & \Z_2 &
0 & 0 \tsep{2pt}\bsep{2pt}\\
\hline
\mbox{\textsf{p31m}} & D_3 & \phi_0 &
\Z_2 & \Z_2 & \Z_2 &
0 & 0 \tsep{2pt}\bsep{2pt}\\
\hline
\mbox{\textsf{p4m}/\textsf{p4g}} & D_4 & \phi_0 &
\Z_2^{\oplus 3} & \Z_2^{\oplus 3} & \Z_2^{\oplus 2} &
0 & \Z_2 \tsep{2pt}\bsep{2pt}\\
\hline
\mbox{\textsf{p4m}/\textsf{p4g}} & D_4 & \phi_1, \phi_2 &
\Z_2^{\oplus 4} & \Z_2^{\oplus 3} & \Z_2^{\oplus 2} &
\Z_2 & \Z_2 \tsep{2pt}\bsep{2pt}\\
\hline
\mbox{\textsf{p6m}} & D_6 & \phi_0 &
\Z_2^{\oplus 2} & \Z_2^{\oplus 2} & \Z_2^{\oplus 2} &
0 & 0 \tsep{2pt}\bsep{2pt}\\
\hline
\mbox{\textsf{p6m}} & D_6 & \phi_1 &
\Z_2^{\oplus 3} & \Z_2^{\oplus 2} & \Z_2^{\oplus 2} &
\Z_2 & 0 \tsep{2pt}\bsep{2pt}\\
\hline
\mbox{\textsf{p6m}} & D_6 & \phi_2 &
\Z_2^{\oplus 3} & \Z_2^{\oplus 2} & \Z_2^{\oplus 2} &
\Z_2 & 0 \tsep{2pt}\bsep{2pt}\\
\hline
\end{array}
$$
\caption{The list of $H^3_P(T^2; \Z_\phi)$ and its subgroups $F^p = F^pH^3_P(T^2; \Z_\phi)$. The $E_\infty$-term of the Leray--Serre spectral sequence is related to these subgroups by $E^{p, 3-p}_\infty \cong F^p/F^{p+1}$.}
\label{fig:list_twisted}
\end{figure}
\end{thm}

It should be noticed that the action of the point group $P$ on the torus relevant to an application of the Freed--Moore $K$-theory to condensed matter physics is the one modif\/ied by a~non-trivial homomorphism $\phi\colon P \to \Z_2$. Some of such modif\/ied actions dif\/fer from those given by the inclusion $P \subset {\rm O}(2)$, and hence are not covered in Theorem~\ref{thm:main_twisted}. The modif\/ied actions should be understood in the context of the so-called \textit{magnetic space groups} (or colour symmetry groups~\cite{Sch2}), and the cohomology as well as the $K$-theory equivariant under the groups deserve to be subjects of a future work.

One may notice that there are more twists for the Freed--Moore $K$-theory in comparison with the twists for equivariant $K$-theory. At present, we lack such an understanding of twists as in Corollary \ref{cor:main}(c) in relation with the nonsymmorphic nature of space groups.

The method for computing $H^3_P(T^2; \Z_\phi)$ and its f\/iltration is similar to the one computing $H^3_P(T^2; \Z)$. In the computation, the cohomology $H^n_P(T^2; \Z_\phi)$ for $n \le 2$ is also determined, as summarized below:

\begin{thm} \label{thm:equiv_coh_twisted}
Let $P$ be the point group of one of the $2$-dimensional space groups $S$, acting on $T^2 = \R^2/\Pi$ via $P \subset {\rm O}(2)$. For $n \le 3$, the $P$-equivariant cohomology $H^n_P(T^2; \Z_\phi)$ with coefficients in the local system $\Z_\phi$ induced from a non-trivial homomorphism $\phi\colon P \to \Z_2$ is as in Fig.~{\rm \ref{fig:list_equiv_coh_twisted}}.
\begin{figure}[htpb]
$$
\begin{array}{|c|c|c||c|c|c|c|}
\hline
\mbox{Space group $S$} & P & \phi &
H^0_P(T^2) & H^1_P(T^2) & H^2_P(T^2) & H^3_P(T^2) \tsep{2pt}\bsep{2pt}\\
\hline
\mbox{\textsf{p2}} & \Z_2 & \phi_1 &
0 & \Z_2 \oplus \Z^{\oplus 2} & 0 & \Z_2^{\oplus 4} \tsep{2pt}\bsep{2pt}\\
\hline
\mbox{\textsf{p4}} & \Z_4 & \phi_1 &
0 & \Z_2 & \Z_2 & \Z_2^{\oplus 2} \tsep{2pt}\bsep{2pt}\\
\hline
\mbox{\textsf{p6}} & \Z_6 & \phi_1 &
0 & \Z_2 & \Z_3 & \Z_2^{\oplus 2} \tsep{2pt}\bsep{2pt}\\
\hline
\mbox{\textsf{pm}/\textsf{pg}} & D_1 & \phi_0 &
0 & \Z_2 \oplus \Z & \Z_2 \oplus \Z & \Z_2^{\oplus 2} \tsep{2pt}\bsep{2pt}\\
\hline
\mbox{\textsf{cm}} & D_1 & \phi_0 &
0 & \Z_2 \oplus \Z & \Z & \Z_2 \tsep{2pt}\bsep{2pt}\\
\hline
\mbox{\textsf{pmm}/\textsf{pmg}/\textsf{pgg}} & D_2 & \phi_0 &
0 & \Z_2 & \Z_2^{\oplus 3} \oplus \Z & \Z_2^{\oplus 4} \tsep{2pt}\bsep{2pt}\\
\hline
\mbox{\textsf{pmm}/\textsf{pmg}/\textsf{pgg}} & D_2 & \phi_1, \phi_2 &
0 & \Z_2 \oplus \Z & \Z_2^{\oplus 2} & \Z_2^{\oplus 6} \tsep{2pt}\bsep{2pt}\\
\hline
\mbox{\textsf{cmm}} & D_2 & \phi_0 &
0 & \Z_2 & \Z_2^{\oplus 2} \oplus \Z & \Z_2^{\oplus 2} \tsep{2pt}\bsep{2pt}\\
\hline
\mbox{\textsf{cmm}} & D_2 & \phi_1, \phi_2 &
0 & \Z_2 \oplus \Z & \Z_2 & \Z_2^{\oplus 4} \tsep{2pt}\bsep{2pt}\\
\hline
\mbox{\textsf{p3m1}} & D_3 & \phi_0 &
0 & \Z_2 & \Z_3^{\oplus 2} \oplus \Z & \Z_2 \tsep{2pt}\bsep{2pt}\\
\hline
\mbox{\textsf{p31m}} & D_3 & \phi_0 &
0 & \Z_2 & \Z_3 \oplus \Z & \Z_2 \tsep{2pt}\bsep{2pt}\\
\hline
\mbox{\textsf{p4m}/\textsf{p4g}} & D_4 & \phi_0 &
0 & \Z_2 & \Z_4 \oplus \Z_2 \oplus \Z & \Z_2^{\oplus 3} \tsep{2pt}\bsep{2pt}\\
\hline
\mbox{\textsf{p4m}/\textsf{p4g}} & D_4 & \phi_1, \phi_2 &
0 & \Z_2 & \Z_2^{\oplus 2} & \Z_2^{\oplus 4} \tsep{2pt}\bsep{2pt}\\
\hline
\mbox{\textsf{p6m}} & D_6 & \phi_0 &
0 & \Z_2 & \Z_6 \oplus \Z & \Z_2^{\oplus 2} \tsep{2pt}\bsep{2pt}\\
\hline
\mbox{\textsf{p6m}} & D_6 & \phi_1 &
0 & \Z_2 & \Z_2 \oplus \Z_3 & \Z_2^{\oplus 3} \tsep{2pt}\bsep{2pt}\\
\hline
\mbox{\textsf{p6m}} & D_6 & \phi_2 &
0 & \Z_2 & \Z_2 & \Z_2^{\oplus 3} \tsep{2pt}\bsep{2pt}\\
\hline
\end{array}
$$
\caption{The list of equivariant cohomology with local coef\/f\/icients.} \label{fig:list_equiv_coh_twisted}
\end{figure}
\end{thm}

Finally, we make comments about the generalizations. To compute cohomology groups of the higher-dimensional tori which are equivariant under space groups, we can in principle apply the three methods in this paper. The f\/irst and second methods would be generalized without dif\/f\/iculty. The third method will however get more dif\/f\/icult, because we need a $P$-CW decomposition of a higher-dimensional torus, which becomes more complicated than decompositions in the $2$-dimensional case. As is suggested by Corollary~\ref{cor:equiv_coh_Z_2}, there are local systems for the Borel equivariant cohomology other than~$\Z_\phi$ associated to a homomorphism $\phi \colon P \to \Z_2$. For the cohomology with such a local system, the notion of reduced cohomology does not make sense. This prevents us from using the second method based on the equivariant stable splitting of the torus, forcing us to use a $P$-CW decomposition.

The outline of this paper is as follows: In Section~\ref{sec:quantum_system_to_K}, we explain how a certain quantum system leads to a twist and def\/ines a twisted $K$-class, mainly based on a formulation in~\cite{F-M}. At the end of this section, a summary of relationship among some natural actions of point groups on tori is included. In Section~\ref{sec:spectral_seq_and_twist}, we review the Leray--Serre spectral sequence for Borel equivariant cohomology and the notion of twists for equivariant $K$-theory. The geometric interpretation of the f\/iltration of the degree $3$ equivariant cohomology is also provided here, after a general property of the spectral sequence is established. Then, in Section~\ref{sec:proof_main_theorem}, we prove Theorems~\ref{thm:main} and~\ref{thm:equiv_coh}. To keep readability of this paper, we provide the detail of computations only in the cases \textsf{p2}, \textsf{p4m/p4g} and \textsf{p6m}. (The detail of the other cases can be found in old versions of \verb|arXiv:1509.09194|.) Section \ref{sec:twisted_case} concerns the equivariant cohomology with the twisted coef\/f\/i\-cient~$\Z_\phi$. We state direct generalizations of some results in the untwisted case, and then prove Theorems~\ref{thm:main_twisted} and~\ref{thm:equiv_coh_twisted}. To keep readability again, we give the details of the computation only in the case of \textsf{p6m} with $\phi_2$. Finally, for convenience, the point group actions of $2$-dimensional space groups are listed in Appendix.

Throughout, familiarity with basic algebraic topology \cite{B-T,Ha} will be supposed.

\section[From quantum systems to twisted $K$-theory]{From quantum systems to twisted $\boldsymbol{K}$-theory}\label{sec:quantum_system_to_K}

We here illustrate how twisted equivariant $K$-theory arises from a quantum system with symmetry, mainly based on a formulation in \cite{F-M}. (We refer the reader to \cite{Th} for a $C^*$-algebraic approach.)

\subsection{Setting}

Let us consider the following mathematical setting:
\begin{itemize}\itemsep=0pt
\item A lattice $\Pi \subset \Pi \otimes_{\Z} \R = \R^d$ of rank $d$.

\item
A subgroup $S$ of the Euclidean group $\R^d \rtimes {\rm O}(d)$ of $\R^d$ which is an extension of a f\/inite group $P \subset {\rm O}(d)$ by $\Pi$:
\begin{gather*}
\begin{array}{c@{}c@{}c@{}c@{}c@{}c@{}c@{}c@{}c@{}c}
1 & \ \longrightarrow \ &
\R^d & \ \longrightarrow \ &
\R^d \rtimes {\rm O}(d) & \ \longrightarrow \ &
{\rm O}(d) & \ \longrightarrow \ &
1 \\
 & &
\cup & &
\cup & &
\cup & &
\\
1 & \ \longrightarrow \ &
\Pi & \ \longrightarrow \ &
S & \ \overset{\pi}{\longrightarrow} \ &
P & \ \longrightarrow \ &
1.
\end{array}
\end{gather*}

\item A unitary representation $U\colon P \to U(V)$ on a f\/inite-dimensional Hermitian vector space~$V$.
\end{itemize}

The group $S$ is nothing but a \textit{$d$-dimensional space group}, and $P$ is called the \textit{point group} of~$S$. When~$S$ is the semi-direct product of $P$ and~$\Pi$, it is called \textit{symmorphic}, otherwise \textit{nonsymmorphic}.

Based on the mathematical setting above, we can introduce a quantum system on $\R^d$ which has $S$ as its symmetry and $V$ as its internal freedom:
\begin{itemize}\itemsep=0pt
\item The `quantum Hilbert space' consisting of `wave functions' is the $L^2$-space $L^2(\R^d, V)$, on which $g \in S$ acts by
$\psi(x) \mapsto (\rho(g)\psi)(x) = U(\pi(g)) \psi(g^{-1} x)$.

\item
The `Hamiltonian' is a self-adjoint operator $H$ on $L^2(\R^d, V)$ invariant under the $S$-action: $H \circ \rho(g) = \rho(g) \circ H$. A typical form of $H$ is $H = \Delta + \Phi$, where $\Delta = \sum \partial^2/\partial x_i^2$ is the Laplacian and $\Phi\colon \R^d \to \mathrm{End}(V)$ is a potential term.
\end{itemize}

\subsection{Bloch transformation}

Even if the Hamiltonian $H$ is invariant under the translation of $\Pi$, a solution $\psi$ to the `time-independent Schr\"{o}dinger equation' $H \psi = E \psi$ with $E \in \R$ is not necessarily $S$-invariant. The so-called `Bloch transformation' allows us to deal with such a situation.

Let $\hat{\Pi} = \operatorname{Hom}(\Pi, U(1))$ denote the Pontryagin dual of the lattice $\Pi$, which is often called the `Brillouin torus' in condensed matter physics. We def\/ine the space $L^2_{\Pi}(\hat{\Pi} \times \R^d, V)$ by
\begin{gather*}
L^2_{\Pi}\big(\hat{\Pi} \times \R^d, V\big)= \big\{ \hat{\psi} \in L^2\big(\hat{\Pi} \times \R^d, V\big) \,|\,
\hat{\psi}(\hat{k}, x + m)= \hat{k}(m) \hat{\psi}(\hat{k}, x) \ (m \in \Pi)
\}.
\end{gather*}
We also def\/ine transformations $\hat{\mathcal{B}}$ and $\mathcal{B}$, inverse to each other:
\begin{alignat*}{3}
&\hat{\mathcal{B}} \colon \ L^2\big(\R^d, V\big) \longrightarrow
L^2_{\Pi}\big(\hat{\Pi} \times \R^d, V\big), \qquad &&
(\hat{\mathcal{B}}\psi)\big(\hat{k}, x\big) = \sum_{n \in \Pi} \hat{k}(n)^{-1} \psi(x + n),& \\
& \mathcal{B} \colon \ L^2_{\Pi}\big(\hat{\Pi} \times \R^d, V\big)
\longrightarrow
L^2\big(\R^d, V\big), \qquad && (\mathcal{B} \hat{\psi})(x)= \int_{\hat{k} \in \hat{\Pi}} \hat{\psi}\big(\hat{k}, x\big) d\hat{k}.&
\end{alignat*}

As is described in \cite{F-M}, the space $L^2_{\Pi}(\hat{\Pi} \times \R^d, V)$ can be identif\/ied with the space $L^2(\hat{\Pi}, \mathcal{E} \otimes V)$ of $L^2$-sections of a vector bundle $\mathcal{E} \otimes V \to \hat{\Pi}$. The inf\/inite-dimensional vector bundle $\mathcal{E} \to \hat{\Pi}$ is given by
\begin{gather*}
\mathcal{E} = \bigcup_{\hat{k} \in \hat{\Pi}} L^2\big(\R^d/\Pi, \mathcal{L}|_{\{ \hat{k} \} \times \R^d/\Pi}\big),
\end{gather*}
where $\mathcal{L} \to \hat{\Pi} \times \R^d/\Pi$ is the Poincar\'{e} line bundle, the quotient of the product line bundle $\hat{\Pi} \times \R^d \times \C \to \hat{\Pi} \times \R^d$ by the following $\Pi$-action
\begin{gather*}
\Pi \times \big(\hat{\Pi} \times \R^d \times \C\big) \longrightarrow \hat{\Pi} \times \R^d \times \C, \qquad
 \big(m, \hat{k}, x, z\big) \mapsto \big(\hat{k}, x + m, \hat{k}(m)z\big).
\end{gather*}
In summary, we get an identif\/ication of $L^2$-spaces
\begin{gather*}
L^2\big(\R^d, V\big) \cong L^2_{\Pi}\big(\hat{\Pi} \times \R^d, V\big) \cong
L^2\big(\hat{\Pi}, \mathcal{E} \otimes V\big).
\end{gather*}

The Hamiltonian $H$ on $L^2(\R^d, V)$ then induces an operator $\hat{H}$ on $L^2_{\Pi}(\hat{\Pi} \times \R^d, V) \cong L^2(\hat{\Pi}, \mathcal{E} \otimes V)$ by $\hat{H} \circ \mathcal{\hat{B}} = \mathcal{\hat{B}} \circ H$. If, for instance, $H$ is of the form $H = \Delta + \Phi$, then $\hat{H}$ preserves the f\/iber of $\mathcal{E} \otimes V$. Generally, this is a consequence of the translation invariance of the Hamiltonian. When the present quantum system is supposed to be an `insulator', a f\/inite number of discrete spectra of $\hat{H}(\hat{k})$ would be conf\/ined to a compact region in $\R$ as $\hat{k} \in \hat{\Pi}$ varies. Then the corresponding eigenfunctions form a f\/inite rank subbundle $E \subset \mathcal{E} \otimes V$, called the `Bloch bundle'. The $K$-class of this vector bundle $E \to \hat{\Pi}$ is regarded as an invariant of the quantum system under study.

\subsection[Nonsymmorphic group and twisted $K$-theory]{Nonsymmorphic group and twisted $\boldsymbol{K}$-theory}

We now take the symmetry into account. From the extension $1 \to \Pi \to S \overset{\pi}{\to} P \to 1$, we can associate a~twisted $P$-equivariant vector bundle on $\hat{\Pi}$ to the $S$-module $L^2(\R^d, V)$. This is a~version of the so-called `Mackey machine'.

Recall that the Euclidean group $\R^d \rtimes {\rm O}(d)$ is the semi-direct product of the orthogonal group ${\rm O}(d)$ and the group of translations $\R^d$. Hence a collection of representatives $\{ s_p \}_{p \in P}$ of $p \in P \cong S/\Pi$ in $S$ is expressed as $s_p = (a_p, p) \in \R^d \rtimes {\rm O}(d)$ by means of a map $a \colon P \to \R^d$. For $p_1, p_2 \in P$ we put
\begin{gather*}
\nu(p_1, p_2) = a_{p_1} + p_1 a_{p_2} - a_{p_1p_2}.
\end{gather*}
Since $\Pi \subset S$ is normal, the action of $P \subset {\rm O}(d)$ on $\R^d$ preserves $\Pi \subset \R^d$. Then we have $\nu(p_1, p_2) \in \Pi$, and $\nu \colon P \times P \to \Pi$ is a group $2$-cocycle of $P$ with values in $\Pi$ regarded as a left $P$-module through the action $m \mapsto pm$ of $p \in P$ on $m \in \Pi$. This group $2$-cocycle measures the failure for $S$ to be symmorphic.

By means of the $S$-action $\rho$ on $L^2(\R^d, V)$, we def\/ine an `action' of $p \in P$ by
\begin{gather*}
\rho(p) \colon \ L^2\big(\R^d, V\big) \longrightarrow L^2\big(\R^d, V\big), \qquad \rho(p) = \rho((a_p, p)),
\end{gather*}
whose explicit formula for $\psi \in L^2(\R^d, V)$ is given by $(\rho(p) \psi)(x) = U(p) \psi(p^{-1} x + a_{p^{-1}})$. The Bloch transformation then induces the following `action' of~$P$,
\begin{gather*}
\hat{\rho}(p) \colon \ L^2_{\Pi}\big(\hat{\Pi} \times \R^d, V\big)
\to L^2_{\Pi}\big(\hat{\Pi} \times \R^d, V\big), \qquad \hat{\rho}(p) \circ \hat{\mathcal{B}} = \hat{\mathcal{B}} \circ \rho(p),
\end{gather*}
whose explicit formula for $\hat{\psi} \in L^2_{\Pi}(\hat{\Pi} \times \R^d, V)$ is
$(\hat{\rho}(p)\hat{\psi})(\hat{k}, x)= U(p) \hat{\psi}(p^{-1}\hat{k}, p^{-1}x + a_{p^{-1}})$.
Here the left $P$-action on $\hat{\Pi}$ is def\/ined by $(p\hat{k})(m) = \hat{k}(p^{-1}m)$, where $p \in P$ acts on $m \in \Pi$ through the inclusion $P \subset {\rm O}(d)$ and the left action of ${\rm O}(d)$ on $\R^d$. Notice that $\rho$ and $\hat{\rho}$ can be honest actions of $P$ in the case of symmorphic $S$, but not in the case of nonsymmorphic $S$, for the usual composition rule is violated:
\begin{gather*}
\big(\hat{\rho}(p_1)\big(\hat{\rho}(p_2)\hat{\psi}\big)\big)\big(\hat{k}, \xi\big)
= \big(p_2^{-1}p_1^{-1}\hat{k}\big)\big(\nu\big(p_2^{-1}, p_1^{-1}\big)\big)
\big(\hat{\rho}(p_1p_2)\hat{\psi}\big)\big(\hat{k}, \xi\big).
\end{gather*}

To interpret the `action' $\hat{\rho}(p)$ in terms of the vector bundle $\mathcal{E} \otimes V$ through $L^2_{\Pi}(\hat{\Pi} \times \R^d, V) \cong L^2(\hat{\Pi}, \mathcal{E} \otimes V)$, recall that the f\/iber of $\mathcal{E} \otimes V$ at $\hat{k} \in \hat{\Pi}$ is
\begin{gather*}
\mathcal{E}|_{\hat{k}} \otimes V = L^2\big(\R^d/\Pi, \mathcal{L}|_{\{ \hat{k} \} \times \R^d/\Pi} \otimes V\big),
\end{gather*}
and $\hat{\psi} \in L^2_{\Pi}(\hat{\Pi} \times \R^d, V)$ corresponds to the following section $\Psi \in L^2(\hat{\Pi}, \mathcal{E} \otimes V)$:
\begin{gather*}
\Psi\big(\hat{k}\big) \colon \
\R^d/\Pi \longrightarrow
\mathcal{L}|_{\{ \hat{k} \} \times \R^d/\Pi} \otimes V, \qquad
x \mapsto \big[\hat{k}, x, \hat{\psi}(\hat{k}, x)\big].
\end{gather*}
Def\/ine for $p \in P$ and $\hat{k} \in \hat{\Pi}$ a linear map
\begin{gather*}
\rho_{\mathcal{E} \otimes V}\big(p; \hat{k}\big) \colon \
\mathcal{E}|_{\hat{k}} \otimes V \longrightarrow
\mathcal{E}|_{p\hat{k}} \otimes V
\end{gather*}
by the assignment of the sections
\begin{gather*}
\rho_{\mathcal{E} \otimes V}\big(p; \hat{k}\big)
\big(\big[x \mapsto \big[\hat{k}, x, \hat{\psi}(\hat{k}, x)\big]\big]\big)
= \big[x \mapsto \big[p\hat{k}, x, U(p) \hat{\psi}\big(\hat{k}, p^{-1}x + a_{p^{-1}}\big)\big]\big].
\end{gather*}
These maps constitute a vector bundle map $\rho_{\mathcal{E} \otimes V}(p)\colon \mathcal{E} \otimes V \to \mathcal{E} \otimes V$ covering the action $\hat{k} \mapsto p\hat{k}$ on $\hat{\Pi}$
\begin{gather*}
\begin{CD}
\mathcal{E} \otimes V @>{\rho_{\mathcal{E} \otimes V}(p)}>>
\mathcal{E} \otimes V \\
@VVV @VVV \\
\hat{\Pi} @>{p}>> \hat{\Pi}.
\end{CD}
\end{gather*}
This is a $\tau$-twisted $P$-action, in the sense that the formula
\begin{gather*}
\rho_{\mathcal{E} \otimes V}\big(p_1; p_2\hat{k}\big)
\rho_{\mathcal{E} \otimes V}\big(p_2; \hat{k}\big) \xi
= \tau\big(p_1, p_2; \hat{k}\big) \rho_{\mathcal{E} \otimes V}\big(p_1p_2; \hat{k}\big) \xi
\end{gather*}
holds for $p_1, p_2 \in P$, $\hat{k} \in \hat{\Pi}$ and $\xi \in \mathcal{E}|_{\hat{k}} \otimes V$. Here $\tau \colon P \times P \times \hat{\Pi} \to U(1)$ is def\/ined by
\begin{gather*}
\tau(p_1, p_2; \hat{k}) = \hat{k}\big(\nu\big(p_2^{-1}, p_1^{-1}\big)\big),
\end{gather*}
and is regarded as a group $2$-cocycle of $P$ with its coef\/f\/icients in the group $C(\hat{\Pi}, U(1))$ of $U(1)$-valued functions on $\hat{\Pi}$ thought of as a right $P$-module through the pull-back under the left action $\hat{k} \mapsto p\hat{k}$ of $p \in P$ on $\hat{k} \in \hat{\Pi}$. The map $\rho_{\mathcal{E} \otimes V}(p)$ on the vector bundle induces the transformation on the sections
\begin{gather*}
\rho_{\mathcal{E} \otimes V}(p) \colon \ L^2\big(\hat{\Pi}, \mathcal{E} \otimes V\big)
\longrightarrow L^2\big(\hat{\Pi}, \mathcal{E} \otimes V\big)
\end{gather*}
by $(\rho_{\mathcal{E} \otimes V}(p)\Psi)(\hat{k}) = \rho_{\mathcal{E} \otimes V}(p; p^{-1}\hat{k})\Psi(p^{-1}\hat{k})$. One can verify that: if $\Psi \in L^2(\hat{\Pi}, \mathcal{E} \otimes V)$ corresponds to $\hat{\psi} \in L^2_{\Pi}(\hat{\Pi} \times \R^d, V)$, then $\rho_{\mathcal{E} \otimes V}(p) \Psi$ corresponds to $\hat{\rho}(p)\hat{\psi}$. Hence the `action' $\hat{\rho}(p)$ on $L^2_\Pi(\hat{\Pi} \times \R^d, V) \cong L^2(\hat{\Pi}, \mathcal{E} \otimes V)$ agrees with the one induced from the $\tau$-twisted $P$-action on $\mathcal{E} \otimes V$.

Now, under the assumption that $\hat{H}$ describes an insulator, the Bloch bundle $E \subset \mathcal{E} \otimes V$ inherits a $\tau$-twisted $P$-action from $\mathcal{E} \otimes V$. This is a consequence of the invariance of the Hamiltonian under the space group action. Therefore the Bloch bundle, being a $\tau$-twisted $P$-equivariant vector bundle of f\/inite rank, def\/ines a class in the $\tau$-twisted $P$-equivariant $K$-theory $K^{\tau + 0}_P(\hat{\Pi})$, which is regarded as an invariant of the insulating system under study.

As is obvious from the construction, we can apply the construction of the group $2$-cocycle $\tau$ to symmorphic space groups. However, in the symmorphic case, the cocycle $\nu$ and hence $\tau$ can be trivialized.

So far a linear representation of $P$ on $V$ is considered. We can relax this representation to be a projective representation of $P$ with its group $2$-cocycle $\omega \colon P \times P \to U(1)$. In this case, the resulting Bloch bundle def\/ines a class in the twisted equivariant $K$-theory $K^{\tau + \omega + 0}_P(\hat{\Pi})$.

\begin{rem} 
The phase factor in the composition rule of $\hat{\rho}$,
\begin{gather*}
\tau_R\big(\hat{k}; p_1, p_2\big)
= \big(p_2^{-1}p_1^{-1}\hat{k}\big)\big(\nu\big(p_2^{-1}, p_1^{-1}\big)\big)
= \hat{k}\big(p_1p_2\nu\big(p_2^{-1}, p_1^{-1}\big)\big)
\end{gather*}
def\/ines a group $2$-cocycle of $P$ with coef\/f\/icients in $C(\hat{\Pi}, U(1))$, when regarded as a left $P$-module by the right action $\hat{k} \mapsto \hat{k}p = p^{-1}\hat{k}$ of $p \in P$ on $\hat{k} \in \hat{\Pi}$. The $2$-cocycles $\tau$ and $\tau_R$ are related by $\tau_R(\hat{k}; p_1, p_2) = \tau(p_1, p_2; (p_1p_2)^{-1}\hat{k})$. This extends to a cochain bijection of group cochains with coef\/f\/icients in the left/right $P$-modules $C(\hat{\Pi}, U(1))$. Thus, $\tau$ and $\tau_R$ have cohomologically the same information. We also remark that $\tau$ and $\tau_R$ are respectively cohomologous to the following $2$-cocycles:
\begin{gather*}
\tau'\big(p_1, p_2; \hat{k}\big) = \big(p_1p_2\hat{k}\big)(\nu(p_1, p_2))^{-1}, \qquad
\tau'_R\big(\hat{k}; p_1, p_2\big) = \hat{k}(\nu(p_1, p_2))^{-1}.
\end{gather*}
\end{rem}

\begin{rem} \label{rem:grading}
Given a homomorphism $c \colon P \to \Z_2$, we can impose that the Hamiltonian~$H$ and the symmetry $\rho(g)$ with $g \in S$ are graded commutative,
$H \circ \rho(g) = c(\pi(g)) \rho(g) \circ H$. Then the quantum system with symmetry in question leads to an element of the twisted equivariant $K$-theory $K^{\tau + c + 0}_P(\hat{\Pi})$, where the (ungraded) twist $\tau$ is now graded by $c \in H^1_P(\hat{\Pi}; \Z_2)$. It should be noticed that the construction of the element uses Karoubi's formulation of $K$-theory~\cite{Kar} and requires a reference quantum system. These points of discussion, which will not be detailed in this paper, are implicit in the absence of the graded twist.
\end{rem}

\begin{rem}
A group $2$-cocycle $\tau$ can be thought of as the cocycle for a projective representation. Besides the argument in this section, there are other arguments which derive projective representations from quantum systems with symmetry (for example~\cite{KKW1, KKW2}).
\end{rem}

\subsection{Actions of the point group on the torus}

To close Section \ref{sec:quantum_system_to_K}, we compare some natural actions of the point group on the torus: Let $S$ be a $d$-dimensional space group, $\Pi$ its lattice, and $P$ its point group.
\begin{enumerate}\itemsep=0pt
\item[(A)]
By the inclusion $P \subset {\rm O}(d)$ and the standard left action of ${\rm O}(d)$ on $\R^d$, the point group $P$ acts on $\R^d$, preserving $\Pi \subset \R^d$. Hence the left action of $P$ on $\R^d$ descends to give a left action of $P$ on the torus $\R^d/\Pi$.

\item[(B)]
By the action (A), the point group $P$ acts on the Pontryagin dual $\hat{\Pi} = \operatorname{Hom}(\Pi, U(1))$ of $\Pi$ from the left: For $p \in P$ and $\hat{k} \in \hat{\Pi}$, we def\/ine $p\hat{k} \in \hat{\Pi}$ by $(p\hat{k})(m) = \hat{k}(p^{-1}m)$ for all $m \in \Pi$.

\item[(C)]
By the inclusion $S \subset \R^d \rtimes {\rm O}(d)$ and the standard left action of $\R^d \rtimes {\rm O}(d)$ on $\R^d$, the space group $S$ acts on $\R^d$. The subgroup $\Pi \subset S$ preserves $\Pi \subset \R^d$, so that the point group $P \cong S/\Pi$ acts on $\R^d/\Pi$.
\end{enumerate}

The action (A) is what we consider in our main results, and the action (B) is relevant to quantum systems as reviewed in this section.

On the one hand, the actions (A) and (B) clearly f\/ix $0 \in \R^d/\Pi$ and $0 \in \hat{\Pi}$, respectively, where we regard $\hat{\Pi}$ as $\operatorname{Hom}(\Pi, \R/\Z)$ via $\R/\Z \cong U(1)$ and $0 \in \hat{\Pi}$ stands for the trivial homomorphism. On the other hand, if $(a_p, p) \in \R^d \rtimes {\rm O}(d)$ is a lift of $p \in P \subset {\rm O}(d)$, then the action of $p \in P$ on $k \in \R^d/\Pi$ in (C) admits the description $k \mapsto pk + a_p$. If $S$ is symmorphic, then we can choose $a_p$ to be in $\Pi$. In this case, the actions (A) and (C) are equivalent. However, if $S$ is nonsymmorphic, then $a_p$ cannot be in $\Pi$. Thus, in this case, the action of $p \in P$ does not f\/ix any point on $\R^d/\Pi$, so that the actions (A) and (C) are not equivalent. For example, in the case of \textsf{pg}, the action of $P = \Z_2$ on the $2$-dimensional torus is free, and its quotient is the Klein bottle.

To compare the actions (A) and (B), we need to identify $\R^d/\Pi$ with $\hat{\Pi} = \operatorname{Hom}(\Pi, \R/\Z)$, which are topologically $d$-dimensional tori. In general, such an identif\/ication may not be unique. A~way to implement the identif\/ication is to choose a basis $\{ v_j \}$ of the lattice $\Pi \cong \Z^d$. This choice induces the following identif\/ications of tori inverse to each other:
\begin{alignat*}{3}
& \hat{\Pi} \longrightarrow \R^d/\Pi, \qquad && \hat{k} \mapsto \sum_j \hat{k}(v_j) v_j,& \\
& \R^d/\Pi \longrightarrow \hat{\Pi}, \qquad && \sum_j k_j v_j \mapsto \bigg[\sum_j m_j v_j \mapsto \sum_j m_jk_j\bigg].&
\end{alignat*}
With this identif\/ication of tori, the left $P$-action on $\hat{\Pi}$ in (B) induces a right $P$-action on $\R^d/\Pi$. Considering the action of $p^{-1}$ instead of $p$, we f\/inally get a left action of $P$ on $\R^d/\Pi$, induced from (B) and the identif\/ication $\hat{\Pi} \cong \R^d/\Pi$. In general, this left action of $P$ on $\R^d/\Pi$ induced from (B) is not equivalent to the action~(A). In the $2$-dimensional case, their relationship is as follows:

\begin{lem}
Let $S$ be a $2$-dimensional space group, $P_S$ its point group, $\Pi_S$ its lattice, and $\hat{\Pi}_S$ the Pontryagin dual of $\Pi_S$.
\begin{itemize}\itemsep=0pt
\item[$(a)$]
We choose a basis $\{ v_j \}$ of $\Pi_S$ to identify $\hat{\Pi}_S$ with $\R^2/\Pi_S$, and let the action in~$(B)$ induce an action of $P_S$ on $\R^2/\Pi_S$. Then, up to equivalence, this action is independent of the choice of the basis.

\item[(b)]
If $S$ is not \textsf{p3m1} or \textsf{p31m}, then the action of $P_S$ on $\R^2/\Pi_S$ induced from~$(B)$ is equivalent to the action of $P_S$ on $\R^2/\Pi_S$ in~$(A)$.

\item[(c)]
If $S$ is \textsf{p3m1} $($respectively \textsf{p31m}$)$, then the action of $P_S$ on $\R^2/\Pi_S$ induced from~$(B)$ is equivalent to the action of $P_{S'}$ on $\R^2/\Pi_{S'}$ in~$(A)$, where $S'$ is \textsf{p31m} $($respectively \textsf{p3m1}$)$.
\end{itemize}
\end{lem}

We remark that the space groups \textsf{p3m1} and \textsf{p31m} share the same lattice and the same point group, as can be seen in Appendix~\ref{sec:appendix:wallpapers}. Hence we have $P_S = P_{S'}$ and $\Pi_S = \Pi_{S'}$ in the third item in the lemma above.

\begin{proof}Considering the action~$(A)$, we def\/ine $\psi(p)_{\ell j} \in \Z$ by $pv_j = \sum_{\ell} \psi(p)_{\ell j}v_{\ell}$, and a homomorphism $\psi\colon P_S \to {\rm GL}(2, \Z)$ by $\psi(p) = (\psi(p)_{\ell j})$. Since $P_S$ is the point group of a $2$-dimensional space group~$S$, the homomorphism $\psi$ is injective and its image $\psi(P_S)$ is a f\/inite subgroup of ${\rm GL}(2, \Z)$. Let $S'$ be another $2$-dimensional space group with its point group~$P'$. Choosing a basis of its lattice $\Pi'$, we similarly get from the action~(A) a homomorphism $\psi'\colon P_{S'} \to {\rm GL}(2, \Z)$. If the images $\psi(P_S)$ and $\psi'(P_{S'})$ are conjugate to each other in ${\rm GL}(2, \Z)$, then the actions of~$P_S$ and~$P_{S'}$ in~(A) are equivalent. The action of $P_S$ on $\R^2/\Pi_S$ induced from~(B) also yields an associated homomorphism $P_S \to {\rm GL}(2, \Z)$. This homomorphism turns out to be the transpose inverse ${}^t \psi^{-1} \colon P_S \to {\rm GL}(2, \Z)$, which is again injective and def\/ines a f\/inite subgroup ${}^t\psi^{-1}(P_S) \subset {\rm GL}(2, \Z)$. If we alter the basis $\{ v_j \}$, then ${}^t\psi^{-1}$ changes by a conjugation of a~matrix in~${\rm GL}(2, \Z)$. Thus, up to conjugations, the image ${}^t\psi^{-1}(P_S) \subset {\rm GL}(2, \Z)$ is independent of the choice of $\{ v_j \}$, showing~(a). Now we can directly verify (b) and (c), by computing the homomorphism $\psi$ based on the explicit basis in Appendix, and comparing the images $\psi(P_S)$ and ${}^t\psi^{-1}(P_S)$ in ${\rm GL}(2, \Z)$.
\end{proof}

Another way of identifying $\hat{\Pi}$ with $\R^d/\Pi$ is to choose a bilinear form $\langle \ , \ \rangle \colon \Pi \times \Pi \to \Z$. We assume that this form is non-degenerate in the sense that the matrix $(\langle v_i, v_j \rangle)$ is invertible with respect to any basis $\{ v_i \}$ of $\Pi$. A non-degenerate bilinear form induces an identif\/ication of the tori as follows
\begin{gather*}
 \R^d/\Pi \longrightarrow \hat{\Pi} = \operatorname{Hom}(\Pi, \R/\Z), \qquad k \mapsto [m \mapsto \langle m, k \rangle].
\end{gather*}
If the bilinear form is $P$-invariant in the sense that $\langle p m, p m' \rangle = \langle m, m' \rangle$ for all $m, m' \in \Pi$ and $p \in P$, then the action (A) on $\R^d/\Pi$ agrees with the action (B) on $\hat{\Pi}$ under the induced identif\/ication $\R^d/\Pi \cong \hat{\Pi}$. For the $2$-dimensional space groups such that $\Pi$ can be the standard lattice $\Z^2 \subset \R^2$, the standard inner product on $\R^2$ restricts to give a $P$-invariant non-degenerate bilinear form. If we choose an orthonormal basis $\{ v_j \}$, then the identif\/ications $\R^d/\Pi \cong \hat{\Pi}$ given by $\langle \ , \ \rangle$ and by $\{ v_j \}$ are $P$-equivariantly the same.

In Section~\ref{sec:proof_main_theorem}, we will work with the torus $\R^2/\Pi$ with the action~(A), and the relation to $\hat{\Pi}$ with the action (B) should be understood as above.

\section{The Leray--Serre spectral sequence and twists}\label{sec:spectral_seq_and_twist}

This section gives a geometric interpretation of the f\/iltration of $H^3_G(X; \Z)$ for the Leray--Serre spectral sequence through types of twists. This is carried out by identifying the Leray--Serre spectral sequence with another natural spectral sequence which computes the Borel equivariant cohomology.

Throughout this section, we assume that $G$ is a f\/inite group acting from the left on a `reasonable' space~$X$, such as a locally contractible, paracompact and regular topological space as in~\cite{FHT}, or a $G$-CW complex~\cite{May}.

\subsection{Spectral sequences}

The Borel equivariant cohomology $H^n_G(X; \Z)$ is def\/ined to be the (singular) cohomology of the quotient space $EG \times_G X$ of $EG \times X$ under the diagonal $G$-action $(\xi, x) \mapsto (\xi g, g^{-1}x)$, where $EG$ is the total space of the universal $G$-bundle $EG \to BG$. Associated to the f\/ibration $X \to EG \times_G X \to BG$ is the Leray--Serre spectral sequence
\begin{gather*}
E_r^{p, q} \Longrightarrow H^{p+q}_G(X; \Z)
\end{gather*}
converging to the graded quotient of a f\/iltration
\begin{gather*}
H^n_G(X; \Z) = F^0H^n_G(X; \Z) \supset F^1H^n_G(X; \Z) \supset \cdots \supset F^{n+1}H^n_G(X; \Z) = 0,
\end{gather*}
that is, $E_\infty^{p, q} = F^pH^{p+q}_G(X; \Z)/F^{p+1}H^{p+q}_G(X; \Z)$. The $E_2$-term is given by the group cohomology of~$G$
\begin{gather*}
E_2^{p, q} = H^p_{\mathrm{group}}(G; H^q(X; \Z)),
\end{gather*}
where the coef\/f\/icient $H^q(X; \Z)$ is regarded as a right $G$-module by the pull-back action. As a~convention of this paper, the group of $p$-cochains with values in a right $G$-module $M$ is denoted by $C^p_{\mathrm{group}}(G; M) = C(G^p, M) = \{ \tau \colon G^p \to M \}$, and the coboundary $\partial \colon C^p_{\mathrm{group}}(G; M) \to C^{p+1}_{\mathrm{group}}(G; M)$ is given by
\begin{gather*}
(\partial \tau)(g_1, \ldots, g_{p+1})= \tau(g_2, \ldots, g_{p+1})+ \sum_{i = 1}^p (-1)^i \tau(g_1, \ldots, g_ig_{i+1}, \ldots, g_{p+1}) \\
\hphantom{(\partial \tau)(g_1, \ldots, g_{p+1})=}{}+ (-1)^{p+1} \tau(g_1, \ldots, g_p)g_{p+1}.
\end{gather*}

As an application of the spectral sequence, we can obtain an identif\/ication $H^n_G(\pt; \Z) \cong H^n_{\mathrm{group}}(G; \Z)$. (We also have $H^n_{\mathrm{group}}(G; \Z) \cong H^{n-1}_{\mathrm{group}}(G; U(1))$ for $n \ge 2$ by the so-called exponential exact sequence.)

For a better geometric understanding of the spectral sequence, let us start with the fact that the Borel equivariant cohomology $H^n_G(X; \Z)$ is isomorphic to the cohomology $H^n(G^\bullet \times X; \Z)$ of a \textit{simplicial space} $G^\bullet \times X$ with its coef\/f\/icients in the constant sheaf $\Z$. This is a consequence of a more general theorem about simplicial space (see \cite{D} for example) together with the fact that the geometric realization $\lvert G^\bullet \times X \rvert$ of $G^\bullet \times X$ is identif\/ied with $EG \times_G X$.

The \textit{simplicial space} $G^\bullet \times X$ is associated to the left $G$-action on $X$, and consists of a sequence of spaces $\{ G^p \times X \}_{p \ge 0 }$ together with the \textit{face map} $\partial_i \colon G^p \times X \to G^{p-1} \times X$, $i = 0, \ldots, p$, and the \textit{degeneracy map} $s_i \colon G^p \times X \to G^{p+1} \times X$, $i = 0, \ldots, p$, given by
\begin{gather*}
\partial_i(g_1, \ldots, g_p, x)=
\begin{cases}
(g_2, \ldots, g_p, x), & i = 0, \\
(g_1, \ldots, g_ig_{i+1}, \ldots, g_p, x), & i = 1, \ldots, p-1, \\
(g_1, \ldots, g_{p-1}, g_px), & i = p,
\end{cases}\\
s_i(g_1, \ldots, g_p, x)=(g_1, \ldots, g_{i-1}, 1, g_{i}, \ldots, g_p, x).
\end{gather*}

The cohomology $H^n(G^\bullet \times X; \Z)$ is then def\/ined to be the total cohomology of the double complex $(C^i(G^j \times X; \Z), \delta, \partial)$, where $(C^i(G^j \times X; \Z), \delta)$ is the complex computing the cohomology of $G^j \times X$ with coef\/f\/icients in $\Z$ and $\partial \colon C^i(G^j \times X; \Z) \to C^i(G^{j+1} \times X; \Z)$ is $\partial = \sum\limits_{i = 0}^{j+1}(-1)^i \partial_i^*$. The double complex admits a natural f\/iltration $\{ \oplus_{j \ge p} C^i(G^j \times X; \Z) \}_{p \ge 0}$. The associated spectral sequence agrees with the Leray--Serre spectral sequence $E_r^{p, q}$, since $G$ is f\/inite.

Now, let us consider the standard exponential exact sequence of sheaves on the simplicial space
$0 \to \Z \to \underline{\R} \to \underline{U(1)} \to 0$, where $\underline{\R}$ consists of the sheaf of $\R$-valued functions on $G^p \times X$ and $\underline{U(1)}$ consists of the sheaf of $U(1)$-valued functions on $G^p \times X$. As in \cite[Lemma~4.4]{G2}, we can readily show that $H^n(G^\bullet \times X; \underline{\R}) = 0$ for $n > 0$. This vanishing together with the associated long exact sequence leads to the following isomorphism for $n \ge 1$
\begin{gather*}
H^{n}(G^\bullet \times X; \underline{U(1)}) \cong H^{n+1}(G^\bullet \times X; \Z).
\end{gather*}
The cohomology $H^{n}(G^\bullet \times X; \underline{U(1)})$ can be def\/ined exactly in the same way as in the case of $H^n(G^\bullet \times X; \Z)$ by using a double complex. Therefore we have a spectral sequence
\begin{gather*}
{}'E^{p, q}_r \Longrightarrow H^{p+q}(G^\bullet \times X; \underline{U(1)})
\end{gather*}
converging to the graded quotient of a f\/iltration
\begin{gather*}
{}'F^0H^n(G^\bullet \times X; \underline{U(1)})
= H^n(G^\bullet \times X; \underline{U(1)}) \supset
{}'F^1H^n(G^\bullet \times X; \underline{U(1)}) \supset
\cdots,
\end{gather*}
whose $E_2$-term is
\begin{gather*}
{}'E_2^{p, q} = H^p_{\mathrm{group}}(G; H^q(X; \underline{U(1)})),
\end{gather*}
where $H^q(X; \underline{U(1)})$ is regarded as a right $G$-module by pull-back. It is clear that $H^0(X; \underline{U(1)}) \cong C(X, U(1))$ and $H^n(X; \underline{U(1)}) \cong H^{n+1}(X; \Z)$ for $n \ge 1$. Since ${}'E_2^{p, q}$ involves the group cohomology with coef\/f\/icients in $C(X, U(1))$, its computation seems to be more complicated than that of~$E_2^{p, q}$. However, the spectral sequence is useful from a geometric viewpoint, as will be seen shortly.

In view of the exponential exact sequence, the f\/iltrations of $H^{n+1}_G(X; \Z) \cong H^{n}(G^\bullet \times X; \underline{U(1)})$ for $n \ge 1$ are related as follows
\begin{gather*}
\begin{array}{c@{}c@{}c@{}c@{}c@{}c@{}c@{}c@{}c@{}c@{}c}
H^n(G^\bullet \times X; \underline{U(1)}) & \ = \ &
{}'F^0H^n & \ \supset \ &
{}'F^{1}H^n & \ \supset \ &
\cdots & \supset \ &
{}'F^{p}H^n & \supset \ &
\cdots \\
\parallel & &
 & &
\downarrow & &
& &
\downarrow & &
 \\
H^{n+1}_G(X; \Z) & \ = \ &
F^0H^{n+1} & \ \supset \ &
F^{1}H^{n+1} & \ \supset \ &
\cdots & \supset \ &
F^{p}H^{n+1} & \ \supset \ &
\cdots.
\end{array}
\end{gather*}
The spectral sequences are related by a map ${}'E^{p, q}_r \to E^{p, q+1}_r$. In particular, the $E_2$-terms ${}'E_2^{p, 0}$ and $E_2^{p, 1}$ are related by a map $C(X, U(1)) \to H^1(X; \Z)$ f\/itting into the exact sequence
\begin{gather*}
0 \to H^0(X; \Z) \to C(X, \R) \to C(X, U(1)) \to H^1(X; \Z) \to 0.
\end{gather*}
As is mentioned, because of the isomorphism $H^q(X; \underline{U(1)}) \cong H^{q+1}(X; \Z)$, we have ${}'E_2^{p, q} \cong E_2^{p, q+1}$ for $q \ge 1$. A more detailed relation between these spectral sequences will be given later under some hypotheses.

\subsection{Twists}\label{subsec:twist}

We here recall the def\/inition of twists for equivariant $K$-theory in \cite{FHT,F-M} for the convenience of the reader. We mainly consider ungraded twists, and refer the reader to \cite{FHT} for the details about graded twists (see also Remark \ref{rem:graded_twist}). Recall that associated to an action of a f\/inite group $G$ on a space $X$ is the \textit{groupoid} $X/\!/G$ such that its set of objects is $X$ and the set of morphisms is $G \times X$.

\begin{dfn}
A \textit{central extension} $(L, \tau)$ of the groupoid $X/\!/G$ consists of the following data:
\begin{itemize}\itemsep=0pt
\item
a Hermitian line bundle $L \to G \times X$, which we write $L_g \to X$ for the restriction to $\{ g \} \times X$ for each $g \in G$,

\item
unitary isomorphisms of Hermitian line bundles $\tau_{g, h} \colon h^*L_g \otimes L_h \to L_{gh}$ on $X$ for each $g, h \in G$, which we write $\tau_{g, h}(x) \colon L_g|_{hx} \otimes L_h|_x \to L_{gh}|_x$ for the restriction to $x \in X$. We assume the following diagram is commutative
\begin{gather*}
\begin{CD}
L_g|_{hkx} \otimes L_h|_{kx} \otimes L_k|_x
@>{1 \otimes \tau_{h, k}(x)}>>
L_g|_{hkx} \otimes L_{hk}|_x \\
@V{\tau_{g, h}(kx) \otimes 1}VV @VV{\tau_{g, hk}(x)}V \\
L_{gh}|_{kx} \otimes L_k|_x @>{\tau_{gh, k}(x)}>> L_{ghk}|_x.
\end{CD}
\end{gather*}
\end{itemize}
\end{dfn}

Notice that if $L_g$ is the product line bundle, then the central extension is just a group $2$-cocycle of $G$ with coef\/f\/icients in $C(X, U(1))$.

\begin{dfn}
An isomorphism $(K, \beta_g) \colon (L_g, \tau_{g, h}) \to (L'_g, \tau'_{g, h})$ of central extensions of $X/\!/G$ consists of the following data:
\begin{itemize}\itemsep=0pt
\item
a Hermitian line bundle $K \to X$,

\item
unitary isomorphisms of Hermitian line bundles $\beta_g \colon L_g \otimes K \to g^*K \otimes L'_g$ on $X$ for each $g \in G$, which we write $\beta_g(x) \colon L_g|_x \otimes K|_x \to K|_{gx} \otimes L'_g|_x$ for the restriction to $x \in X$. We assume the following diagram is commutative
\begin{gather*}
\begin{CD}
L_g|_{hx} \otimes L_h|_x \otimes K|_x @>{1 \otimes \beta_h(x)}>>
L_g|_{hx} \otimes K|_{hx} \otimes L'_h|_x \\
@V{\tau_{g, h}(x) \otimes 1}VV @VV{\beta_g(hx) \otimes 1}V \\
L_{gh}|_x \otimes K|_x @. K|_{ghx} \otimes L'_g|_{hx} \otimes L'_h|_x \\
@| @VV{1 \otimes \tau'_{g, h}(x)}V \\
L_{gh}|_x \otimes K|_x @>{\beta_{gh}(x)}>> K|_{ghx} \otimes L'_{gh}|_x.
\end{CD}
\end{gather*}

\end{itemize}
The isomorphisms $(K, \beta_g)$ and $(K', \beta'_g)$ from $(L_g, \tau_{g, h})$ to $(L'_g, \tau'_{g, h})$ are identif\/ied if there is a~unitary isomorphism $f \colon K \to K'$ making the following diagram commutative
\begin{gather*}
\begin{CD}
L_g|_x \otimes K|_x @>{\beta_g(x)}>> K|_{gx} \otimes L'_g|x \\
@V{1 \otimes f(x)}VV @VV{f(gx) \otimes 1}V \\
L_g|_x \otimes K'|_x @>{\beta'_g(x)}>> K'|_{gx} \otimes L'_g|x.
\end{CD}
\end{gather*}
\end{dfn}

\begin{dfn}
An \textit{ungraded $G$-equivariant twist of $X$}, or a \textit{twist} for short, is a central extension of a groupoid $\tilde{X}$ which has a local equivalence to $X/\!/G$.
\end{dfn}

A point in this def\/inition is that a twist needs an extra groupoid $\tilde{X}$. A central extension of~$X/\!/G$ is a special type of a twist such that $\tilde{X} = X/\!/G$. Taking the extra groupoids into account, we can introduce a notion of isomorphisms to twists. We refer the reader to~\cite{FHT} for the details of the isomorphisms and the following classif\/ication:

\begin{prop}[\cite{FHT}]
The isomorphisms classes of ungraded $G$-equivariant twists of $X$ form an abelian group isomorphic to $H^3_G(X; \Z)$.
\end{prop}

A key to the classif\/ication is the isomorphism $H^3_G(X; \Z) \cong H^2(G^\bullet \times X; \underline{U(1)})$. A close look at the proof of the classif\/ication leads to:

\begin{lem} \label{lem:twist_interpretation_preliminary}
The following holds true:
\begin{itemize}\itemsep=0pt
\item[$(i)$]
${}'F^1H^2(G^\bullet \times X; \underline{U(1)})$ classifies twists represented by central extensions of the groupoid $X/\!/G$.

\item[$(ii)$]
${}'F^2H^2(G^\bullet \times X; \underline{U(1)})$ classifies twists represented by group $2$-cocycles of $G$ with coefficients in the $G$-module $C(X, U(1))$.
\end{itemize}
\end{lem}

\begin{rem}
In \cite{F-M}, an isomorphism of central extensions of $X/\!/G$ is formulated only by using the product line bundle $K = X \times \C$. The reason of the dif\/ference in these def\/initions is that we are considering an isomorphism of central extensions of $X/\!/G$ regarded as twists. By the same reasoning, group cocycles which are not cohomologous to each other can be isomorphic as twists.
\end{rem}

\begin{rem} \label{rem:graded_twist}
The modif\/ication needed to def\/ine a graded twist is to replace the Hermitian line bundle $L$ constituting a central extension $(L, \tau)$ with a $\Z_2$-graded Hermitian line bundle. Since~$L$ is of rank~$1$, its $\Z_2$-grading amounts to specifying the degree of $L$ to be even or odd. With the suitable modif\/ication of the notion of isomorphisms, we can eventually classify graded twists by $H^1_G(X; \Z_2) \times H^3_G(X; \Z)$.
\end{rem}

\subsection{Comparison of two spectral sequences}\label{subsec:comparison}

The relation between the spectral sequences $E_r^{p, q}$ and ${}'E_r^{p, q}$ can be made more clear under a simple assumption. To present this here, we begin with a key lemma: Recall that the exponential exact sequence of sheaves on $X$ induces a natural exact sequence of right $G$-modules
\begin{gather*}
0 \to H^0(X; \Z) \to C(X, \R) \to C(X, U(1)) \to H^1(X; \Z) \to 0.
\end{gather*}
Let us fold this into a short exact sequence
\begin{gather*}
0 \to C(X, \R)/H^0(X; \Z) \to C(X, U(1)) \to H^1(X; \Z) \to 0.
\end{gather*}
In general, this does not split as an exact sequence of $G$-modules. (Such an example is provided by the circle $S^1 \subset \R^2$ with the action of $D_2 \subset {\rm O}(2)$.) Notice that if $X$ is path connected, then $H^0(X; \Z) = \Z$.

\begin{lem} \label{lem:splitting}
If a finite group $G$ acts on a compact and path connected space $X$ fixing a point $\pt \in X$, then the following exact sequence of $G$-modules splits
\begin{gather*}
0 \to C(X, \R)/\Z \to C(X, U(1)) \to H^1(X; \Z) \to 0.
\end{gather*}
\end{lem}

\begin{proof}
For notational convenience, we use the identif\/ication $U(1) \cong \R/\Z$ in this proof. Let $C(X, \pt, \R) \subset C(X, \R)$ be the subgroup consisting of functions taking $0$ at $\pt$. The inclusion $\iota \colon \pt \to X$ induces an isomorphism of $G$-modules
\begin{gather*}
C(X, \R) \longrightarrow C(X, \pt, \R) \oplus \R, \qquad f \mapsto (f - \iota^*f, \iota^*f).
\end{gather*}
Similarly, we have an isomorphism $C(X, \R/\Z) \cong C(X, \pt, \R/\Z) \oplus \R/\Z$ of $G$-modules. Thus the exact sequence of $G$-modules in question is equivalent to
\begin{gather*}
0 \to C(X, \pt, \R) \to C(X, \pt, \R/\Z) \overset{\delta}{\to} H^1(X; \Z) \to 0.
\end{gather*}
Since $X$ is supposed to be compact, $H^1(X; \Z)$ is a free abelian group of f\/inite rank. Let us choose a basis $H^1(X; \Z) \cong \bigoplus_i \Z a_i$, and also $\varphi_i \colon X \to \R/\Z$ such that $\delta \varphi_i = a_i$ and $\varphi_i(\pt) = 0$. Modifying the splitting $a_i \mapsto \varphi_i$ of the exact sequence of \textit{abelian groups}, we construct a splitting of the exact sequence of \textit{$G$-modules}, which will complete the proof.

For the modif\/ication, we introduce a square matrix $A(g) = (A_{ij}(g))$ with integer coef\/f\/icients to each $g \in G$ by $g^*a_i = \sum_j A_{ij}(g) a_j$. It holds that $A(gh) = A(g)A(h)$. Because of the exact sequence, there are functions $f^i_g \in C(X, \pt, \R)$ such that the following holds in $C(X, \pt, \R/\Z)$:
\begin{gather*}
g^*\varphi_i = \sum_j A_{ij}(g) \varphi_j + \big(f_g^i \ \mathrm{mod} \ \Z\big).
\end{gather*}
This can be expressed as $g^*\Phi = A(g) \Phi + F_g$ by using the vectors $\Phi = (\varphi_i)$ and $F_g = (F^i_g)$. It then holds that $F_{gh} = A(g) F_h + h^*F_g$ in $C(X, \pt, \R)$. Since $A(g)$ is invertible, this is equivalent to
\begin{gather*}
A(gh)^{-1}F_{gh} = A(h)^{-1} F_h + A(h)^{-1} h^*\big(A(g)^{-1}F_g\big).
\end{gather*}
Write $\lvert G \rvert$ for the order of $G$, and put $\overline{F} = \frac{1}{\lvert G \rvert}\sum\limits_{g \in G} A(g)^{-1} F_g$. Taking the average over $g \in G$ in the formula above, we get
\begin{gather*}
\overline{F} = A(h)^{-1} F_h + A(h)^{-1} h^*\overline{F},
\end{gather*}
which is equivalent to $F_g = A(g) \overline{F} - g^*\overline{F}$. Now $g^*(\Phi + \overline{F}) = A(g) (\Phi + \overline{F})$. Thus, under the expression $\overline{F} = ( \overline{f}^i )$ by using $\overline{f}^i \in C(X, \pt, \R)$, the assignment $a_i \mapsto \varphi_i + (\overline{f}^i \ \mathrm{mod} \ \Z)$ def\/ines a~splitting $H^1(X; \Z) \to C(X, \pt, \R/\Z)$ compatible with the $G$-module structures.
\end{proof}

\begin{lem} \label{lem:splitting_group_coh}
Let $G$ be a finite group acting on a compact and path connected space $X$ fixing a~point $\pt \in X$. Then, for $n \ge 1$, there is an isomorphism
\begin{gather*}
H^n_{\mathrm{group}}(G; C(X, U(1))) \cong H^n_{\mathrm{group}}(G; U(1)) \oplus H^n_{\mathrm{group}}\big(G; H^1(X; \Z)\big),
\end{gather*}
where $U(1)$ is the trivial $G$-module, and $H^1(X; \Z)$ is regarded as a $G$-module through the action of $G$ on $X$.
\end{lem}

\begin{proof}
Lemma \ref{lem:splitting} implies
\begin{gather*}
H^n_{\mathrm{group}}(G; C(X, U(1))) \cong H^{n}_{\mathrm{group}}(G; C(X, \R)/\Z) \oplus H^n_{\mathrm{group}}\big(G; H^1(X; \Z)\big)
\end{gather*}
for all $n \ge 1$. By the $G$-module isomorphism $C(X, \R)/\Z \cong C(X, \pt, \R) \oplus \R/\Z$ utilized in Lem\-ma~\ref{lem:splitting}, we have
\begin{gather*}
H^n_{\mathrm{group}}(G; C(X, \R)/\Z) \cong H^n_{\mathrm{group}}(G; C(X, \pt, \R)) \oplus H^n_{\mathrm{group}}(G; \R/\Z).
\end{gather*}
Since $C(X, \pt, \R)$ is a vector space over $\R$, we can prove the vanishing $H^n_{\mathrm{group}}(G; C(X, \pt, \R)) = 0$
for $n \ge 1$ by an average argument as in \cite[Lemma 4.4]{G2}.
\end{proof}

\begin{prop} \label{prop:splitting_spectral_seq}
Suppose that a finite group $G$ acts on a compact and path connected space~$X$ fixing a point $\pt \in X$. Then for $r \ge 2$ we have
\begin{gather*}
{}'E^{p, 0}_r \cong E^{p, 1}_r \oplus E^{p+1, 0}_r, \quad p \ge 1, \qquad {}'E^{p, q}_r \cong E^{p, q+1}_r, \quad p \ge 0, q \ge 1.
\end{gather*}
\end{prop}

\begin{proof}
Recall that the exponential exact sequence induces the connecting homomorphism $\delta \colon H^q(X; \underline{U(1)}) \to H^{q+1}(X; \Z)$ and this induces a natural homomorphism $\delta \colon {}'E^{p, q}_r \to E^{p, q+1}_r$ compatible with the dif\/ferentials ${}'d_r$ and $d_r$. In the case of $r = 2$, the homomorphism $\delta \colon {}'E_2^{p, q} \to E^{p, q+1}_2$ is bijective for $q \ge 1$ and $p \ge 0$, and we have ${}'E^{p, 0}_2 \cong E^{p, 1}_2 \oplus E^{p+1, 0}_2$ for $p \ge 1$ as a~consequence of Lemma~\ref{lem:splitting_group_coh}. Notice that, under this isomorphism, $\delta \colon {}'E^{p, 0}_2 \to E^{p, 1}_2$ for $p \ge 1$ restricts to the identity on the direct summand $E^{p, 1}_2 \subset {}'E^{p, 0}_2$. Note also that $E^{p, 0}_2 = E^{p, 0}_\infty$ for any~$p$, because
\begin{gather*}
E^{p, 0}_2 = H^p_{\mathrm{group}}(G; \Z) = H^p(BG; \Z) = H^p_G(\mathrm{pt}; \Z)
\end{gather*}
is a direct summand of $H^p_G(X; \Z) \cong H^0_G(\pt; \Z) \oplus \tilde{H}^p_G(X; \Z)$, where $\tilde{H}^p_G(X; \Z)$ is the reduced cohomology. Thus, for $p \ge 1$, the map $\delta \colon {}'E^{p, 0}_2 \to E^{p, 1}_2$ is the projection onto $E^{p, 1}_2$ and the image of the dif\/ferential ${}'d_2 \colon {}'E^{p-2, 1}_2 \to {}'E_2^{p, 0}$ is in the direct summand $E^{p, 0}_2$. This leads to
\begin{gather*}
{}'E^{p, 0}_3 \cong E^{p, 0}_3 \oplus E^{p, 1}_3, \quad p \ge 1, \qquad {}'E^{p, q}_3 \cong E^{p, q+1}_3, \quad p \ge 0, q \ge 1.
\end{gather*}
The calculation above can be repeated inductively on~$r$.
\end{proof}

\begin{cor} \label{cor:comparison_filtration}
Let $G$ and $X$ be as in Proposition~{\rm \ref{prop:splitting_spectral_seq}}. Then, for any $n \ge 1$ and $p = 0, \ldots, n$, there is a natural isomorphism
\begin{gather*}
{}'F^pH^n(G^\bullet \times X; \underline{U(1)}) \cong F^pH^{n+1}_G(X; \Z).
\end{gather*}
In addition, we have a decomposition
\begin{gather*}
{}'F^nH^n(G^\bullet \times X\colon \underline{U(1)}) \cong F^nH^{n+1}_G(X; \Z) \cong E^{n, 1}_\infty \oplus F^{n+1}H^{n+1}_G(X; \Z),
\end{gather*}
in which $F^{n+1}H^{n+1}_G(X; \Z) \cong H^{n+1}_G(\pt; \Z) \cong H^n_{\mathrm{group}}(G; U(1))$.
\end{cor}

\begin{proof}
Put ${}'F^pH^n = {}'F^pH^n(G^\bullet \times X; \underline{U(1)})$ and $F^pH^{n+1} = F^pH^{n+1}_G(X; \Z)$ for short. The exponential exact sequence induces a homomorphism of short exact sequences
\begin{gather*}
\begin{CD}
0 @>>>
{}'F^{p+1}H^n @>>>
{}'F^pH^n @>>>
{}'E^{p, n-p}_\infty @>>>
0 \\
@. @V{\delta}VV @V{\delta}VV @V{\delta}VV @. \\
0 @>>>
F^{p+1}H^{n+1} @>>>
F^pH^{n+1} @>>>
E^{p, n+1-p}_\infty @>>>
0.
\end{CD}
\end{gather*}
In the case of $p = n$, the diagram above becomes
\begin{gather*}
\begin{CD}
0 @>>>
0 @>>>
{}'F^nH^n @>{\cong}>>
{}'E^{n, 0}_\infty @>>>
0 \\
@. @V{\delta}VV @V{\delta}VV @V{\delta}VV @. \\
0 @>>>
F^{n+1}H^{n+1} @>>>
F^nH^{n+1} @>>>
E^{n, 1}_\infty @>>>
0.
\end{CD}
\end{gather*}
Notice that $F^{n+1}H^{n+1} \cong E^{n+1, 0}_\infty \cong E^{n+1, 0}_2$ since $E^{n+1, 0}_2 \cong H^{n+1}_G(\pt; \Z)$ must survive into $H^{n+1}_G(X; \Z) \cong H^{n+1}_G(\pt; \Z) \oplus \tilde{H}^{n+1}_G(X; \Z)$. Hence $F^nH^{n+1} \cong E^{n+1, 0}_\infty \oplus E^{n, 1}_\infty$. If $n \ge 1$, then this isomorphism is compatible with the isomorphism ${}'E^{n, 0}_\infty \cong E^{n + 1, 0}_\infty \oplus E^{n, 1}_\infty$ in Proposition~\ref{prop:splitting_spectral_seq} through $\delta$, so that
\begin{gather*}
{}'F^nH^n \overset{\delta}{\cong} F^nH^{n+1} \cong F^{n+1}H^{n+1} \oplus E^{n, 1}_\infty.
\end{gather*}
For $p = n-1, n-2, \ldots, 1, 0$, we know $\delta \colon {}'E^{p, n-p}_\infty \to E^{p, n-p+1}_\infty$ is bijective by Proposition~\ref{prop:splitting_spectral_seq}. Therefore ${}'F^pH^n \cong F^pH^{n+1}$ inductively.
\end{proof}

Combining the above corollary with Lemma \ref{lem:twist_interpretation_preliminary}, we get the interpretations of $F^pH^3_G(X; \Z)$ by twists presented in Introduction:

\begin{cor} \label{cor:twist_interpretation}
Let $G$ and $X$ be as in Proposition~{\rm \ref{prop:splitting_spectral_seq}}.
\begin{itemize}\itemsep=0pt
\item[$(i)$]
$F^1H^3_G(X; \Z)$ classifies twists which can be represented by central extensions of the groupoid $X/\!/G$.

\item[$(ii)$]
$F^2H^3_G(X; \Z)$ classifies twists which can be represented by $2$-cocycles of $G$ with coefficients in the $G$-module $C(X, U(1))$.

\item[$(iii)$]
$F^3H^3_G(X; \Z) = H^2_{\mathrm{group}}(G; U(1))$ classifies twists which can be represented by $2$-cocycles of~$G$ with coefficients in the trivial $G$-module $U(1)$.
\end{itemize}
\end{cor}

\begin{rem}
The coincidence ${}'F^1H^n(G^\bullet \times X; \underline{U(1)}) = F^1H^{n+1}(X; \Z)$ in Corollary~\ref{cor:comparison_filtration} holds true for $n \ge 0$ without the assumption that~$G$ f\/ixes a point on~$X$. This is because~${}'E^{0, n}$ and~$E^{0, n+1}$ are subgroups of $H^n(X; \underline{U(1)}) \cong H^{n+1}(X; \Z)$ and it holds that
\begin{align*}
{}'F^1H^n(G^\bullet \times X; \underline{U(1)})&= F^1H^{n+1}(X; \Z) \\
&= \operatorname{Ker}\big[f \colon H^n(G^\bullet \times X; \underline{U(1)})\cong H^{n+1}_G(X; \Z) \to H^{n+1}(X; \Z)\big],
\end{align*}
where $f$ is the homomorphism of ``forgetting the group actions''.
\end{rem}

\section{The proof of Theorems \ref{thm:main} and \ref{thm:equiv_coh}}\label{sec:proof_main_theorem}

Theorems \ref{thm:main} and~\ref{thm:equiv_coh} are proved here based on case-by-case computations of the equivariant cohomology and the Leray--Serre spectral sequence. Some basic facts that are useful for the computation are summarized f\/irst. We then provide the outline of the computations and details of typical cases, \textsf{p2}, \textsf{p4m/p4g} and \textsf{p6m}. Finally, Corollary~\ref{cor:main} is proved.

\subsection{Some generality}\label{subsec:some_generality}

The cohomology $H^n(T^2; \Z)$ of the torus is well known, so that nothing remains to be proven in the case of \textsf{p1}.

For the point group $P$ of any $2$-dimensional space group, the vanishing $H^3(T^2; \Z) = 0$ implies $E^{0, 3}_\infty = 0$, so that
\begin{gather*}
H^3_P\big(T^2; \Z\big) = F^0H^3_P\big(T^2; \Z\big) = F^1H^3_P\big(T^2; \Z\big).
\end{gather*}
Note that each point group $P$ f\/ixes a point on $T^2$, so that
\begin{gather*}
F^3H^3_P\big(T^2; \Z\big) = H^3_P(\pt; \Z)= H^3_{\mathrm{group}}(P; \Z)= H^2_{\mathrm{group}}(P; U(1)).
\end{gather*}
Then the main task for the proof of Theorem \ref{thm:main} is to compute $H^3_P(T^2; \Z)$ and $F^2H^3_P(T^2; \Z)$, since in the case where $P$ is the cyclic group $\Z_n$ or the dihedral group $D_n$, the cohomology $H^m_P(\pt; \Z)$ is summarized as follows:
$$
\begin{array}{|c|c|c|c|c|}
\hline
P & H^0_P(\pt; \Z) & H^1_P(\pt; \Z) & H^2_P(\pt; \Z) & H^3_P(\pt; \Z) \tsep{2pt}\bsep{2pt}\\
\hline
\Z_n & \Z & 0 & \Z_n & 0 \\
\hline
D_n & \Z & 0 &
\begin{cases}
\Z_2 & (\mbox{$n$: odd}) \\
\Z_2^{\oplus 2} & (\mbox{$n$: even})
\end{cases} &
\begin{cases}
0 & (\mbox{$n$: odd}) \\
\Z_2 & (\mbox{$n$: even})
\end{cases}\tsep{10pt}\bsep{10pt}\\
\hline
\end{array}
$$
The degree $0$ part $H^0_P(\pt; \Z) = H^0(BP; \Z) = \Z$ is clear. Since $P$ is f\/inite, the degree $1$ part $H^1_P(\pt; \Z) \cong \operatorname{Hom}(P, \Z)$ is trivial. The degree $2$ part $H^2_P(\pt; \Z) \cong \operatorname{Hom}(P, U(1))$ can be seen by the classif\/ication of irreducible representations. Finally, the degree $3$ part $H^3_P(\pt; \Z) \cong H^2_{\mathrm{group}}(P; U(1))$ for $P = \Z_n, D_n$ can be found in \cite{Ka}.

In the rest of the section, we may use a structure of $T^2$ as a $P$-CW complex. In general, for a compact Lie group $G$, a \textit{$G$-CW complex} is an analogue of a CW complex made of \textit{$G$-cells}. A~$d$-dimensional $G$-cell is a $G$-space of the form $G/H \times e^d$, where $H \subset G$ is a closed subgroup and~$e^d$ is the standard $d$-dimensional cell. The $G$-action on $G/H$ is the left translation, whereas that on~$e^d$ is trivial. For the details, we refer the reader to \cite{May}.

We later compute a group cohomology via cohomology of a space:

\begin{lem} \label{lem:1_dim_complex}
Suppose that a finite group $G$ acts on a path connected space $Y$ fixing at least one point $\pt \in Y$. Suppose further that $Y$ is a CW complex consisting of only cells of dimension less than or equal to $1$. Then the following holds true for all $n \ge 0$.
\begin{gather*}
H^n_{\mathrm{group}}\big(G; H^1(Y; \Z)\big) \cong \tilde{H}^{n+1}_G(Y; \Z),
\end{gather*}
where $\tilde{H}^{n+1}_G(Y; \Z)$ stands for the reduced cohomology.
\end{lem}

Notice that a $G$-CW complex is naturally a CW complex.

\begin{proof}
Consider the Leray--Serre spectral sequence
\begin{gather*}
E_2^{p, q} = H^p_{\mathrm{group}}(G; H^q(Y; \Z))\Longrightarrow H^*_G(Y; \Z).
\end{gather*}
Note that $H^q(Y; \Z) = 0$ for $q \neq 0, 1$. The $E_2$-term $E_2^{p, 0} = H^p_{\mathrm{group}}(G; \Z)$ must survive into the direct summand $H^p_G(\pt; \Z)$ in $H^p_G(Y; \Z) \cong H^p_G(\pt; \Z) \oplus \tilde{H}^p_G(Y; \Z)$. Therefore it must hold that
\begin{gather*}
\tilde{H}^p_G(Y; \Z) \cong E^{p-1, 1}_\infty = E^{p-1, 1}_2\cong H^{p-1}_{\mathrm{group}}\big(G; H^1(Y; \Z)\big),
\end{gather*}
which completes the proof.
\end{proof}

We also prepare a simple lemma about group cohomology: Let $G$ be a f\/inite group, $c \colon G \to \Z_2 = \{ \pm 1 \}$ a surjective homomorphism, and $\tilde{\Z} = \Z_c$ the $G$-module such that its underlying group is $\Z$ and $G$ acts (from the right) by $m \mapsto m c(g)$. A typical example is a f\/inite subgroup $P \subset {\rm O}(2)$ such that $P \not\subset {\rm SO}(2)$ with $c$ the composition of the inclusion $P \to {\rm O}(2)$ and the determinant ${\rm O}(2) \to \Z_2$.

\begin{lem} \label{lem:group_coh_orientation_coeff}
Let $G$, $c$ and $\tilde{\Z}$ be as above. Then, $H^0_{\mathrm{group}}(G; \tilde{\Z}) = 0$ and $H^1_{\mathrm{group}}(G; \tilde{\Z}) \cong \Z_2$.
\end{lem}

\begin{proof}For any $n \in C^0_{\mathrm{group}}(G; \tilde{\Z}) = \Z$, its coboundary $\partial n\colon G \to \Z$ is $(\partial n)(g) = n(1 - c(g))$. Thus, the assumption that $c$ is surjective implies the vanishing of the $0$th cohomology. The inclusion $\operatorname{Ker}(c) \subset G$ induces an injection on $1$-cocycles
\begin{gather*}
Z^1_{\mathrm{group}}\big(G; \tilde{\Z}\big) \to Z^1_{\mathrm{group}}\big(\operatorname{Ker}(c); \tilde{\Z}\big) = \operatorname{Hom}(\operatorname{Ker}(c), \Z) = 0.
\end{gather*}
Thus, given a group $1$-cocycle $\phi \in Z^1_{\mathrm{group}}(G; \tilde{\Z})$, it holds that $\phi(g) = 0$ for all $g \in \operatorname{Ker}(c)$. If $g, h \not\in \operatorname{Ker}(c)$, then the cocycle condition $(\partial \phi)(g, h) = 0$ implies $\phi(g) = \phi(h)$. Therefore $\phi \colon G \to \Z$ is always of the form $\phi(g) = n (1 - c(g))/2$ for some $n \in \Z$. This provides the identif\/ication $Z^1_{\mathrm{group}}(G; \tilde{\Z}) \cong \Z$ as well as $B^1_{\mathrm{group}}(G; \tilde{\Z}) \cong 2\Z$, which completes the proof.
\end{proof}

In some cases, the computations of the Leray--Serre spectral sequence are similar, which we summarize as follows:

\begin{lem} \label{lem:typical_argument_non_orientation_preserving_case}
Let $G$ be a finite group acting on the torus $T^2$ such that:
\begin{itemize}\itemsep=0pt
\item there is a fixed point $\pt \in T^2$,

\item the $G$-action does not preserve the orientation of $T^2$.
\end{itemize}
Then the following holds true about the Leray--Serre spectral sequence:
\begin{itemize}\itemsep=0pt
\item[$(a)$] $F^2H^3_G(T^2; \Z) \cong E_2^{2, 1} \oplus E_2^{3, 0}$,

\item[$(b)$] $H^n_G(T^2; \Z) \cong \bigoplus_{p + q = n} E_2^{p, q}$ for $n \le 2$.
\end{itemize}
\end{lem}

\begin{proof}
In the Leray--Serre spectral sequence
$E_2^{p, q} = H^p_{\mathrm{group}}(G; H^q(T^2; \Z))$, the coef\/f\/icient in the group cohomology $H^0(T^2) \cong \Z$ is identif\/ied with the trivial $G$-module, and $H^2(T^2) \cong \Z$ with the $G$-module in Lemma \ref{lem:group_coh_orientation_coeff}. Then the relevant $E_2$-terms can be summarized as follows:
$$
\begin{array}{|c|c|c|c|c|c|}
\hline
q = 3 & 0 & 0 & 0 & 0 & 0 \\
\hline
q = 2 & 0 & \Z_2 & & & \\
\hline
q = 1 & E_2^{0, 1} & E_2^{1, 1} & E_2^{2, 1} & & \tsep{2pt}\bsep{2pt}\\
\hline
q = 0 & E_2^{0, 0} & E_2^{1, 0} & E_2^{2, 0} & E_2^{3, 0} & E_2^{4, 0} \tsep{2pt}\bsep{2pt}\\
\hline
E_2^{p, q} \tsep{2pt}\bsep{2pt} & p = 0 & p = 1 & p = 2 & p = 3 & p = 4\\
\hline
\end{array}
$$
Since $G$ f\/ixes $\pt \in T^2$, we have the decomposition $H^n_G(T^2; \Z) \cong H^n_G(\pt; \Z) \oplus \tilde{H}^n_G(T^2; \Z)$, where $\tilde{H}^n_G(T^2; \Z)$ is the reduced cohomology. Therefore the $E_2$-term $E_2^{n, 0} = H^n_{\mathrm{group}}(G; \Z) \cong H^n_G(\pt; \Z)$ must survive into the direct summand $H^n_G(\pt; \Z)$ in $H^n_G(T^2; \Z)$. This implies that $E_2^{n, 0} = E_\infty^{n, 0}$ is always a direct summand of the subgroups $F^pH^n_G(T^2; \Z) \subset H^n_G(T^2; \Z)$ and that $d_2 \colon E_2^{p-2, 1} \to E_2^{p, 0}$ is trivial. As a result, we get $E_2^{2, 1} = E_\infty^{2, 1}$ and the isomorphism (a). Also $E_2^{p, q} = E_\infty^{p, q}$ for $p + q \le 2$, and the isomorphism (b) follows.
\end{proof}

The degeneration of the spectral sequence in the above lemma can be generalized in some cases. For this aim, the key is the following equivariant stable splitting of~$T^2$ (cf.\ \cite[Theorem~11.8]{F-M}).

\begin{lem} \label{lem:stable_splitting}
Suppose a finite group $G$ acts on the torus $T^2 = S^1 \times S^1$ and
\begin{itemize}\itemsep=0pt
\item there is a fixed point $\pt = (x_0, y_0) \in T^2$ under the $G$-action,

\item $G$ preserves the subspace $S^1 \vee S^1 = S^1 \times \{ y_0 \} \cup \{ x_0 \} \times S^1 \subset T^2$.
\end{itemize}
Then there is a $G$-equivariant homotopy equivalence
\begin{gather*}
\Sigma T^2 \simeq \Sigma\big(S^1 \vee S^1\big) \vee \Sigma \big(T^2/S^1\vee S^1\big),
\end{gather*}
where $\Sigma$ stands for the reduced suspension.
\end{lem}

\begin{proof}
The argument of the proof of Proposition~4I.1 \cite[p.~467]{Ha} can be applied to our equivariant case.
\end{proof}

We remark that the point groups of the $2$-dimensional space groups without elements of order~$3$ fulf\/ill the assumptions of the lemma above.

\begin{lem} \label{lem:equiv_coh_splitting_case}
Under the assumption in Lemma {\rm \ref{lem:stable_splitting}}, we have the following isomorphism of groups for all $n \in \Z$
\begin{gather*}
H^n_G\big(T^2; \Z\big) \cong H^n_G(\pt; \Z) \oplus \tilde{H}^n_G\big(S^1 \vee S^1; \Z\big) \oplus \tilde{H}^n_G\big(T^2/S^1 \vee S^1; \Z\big).
\end{gather*}
Further, the Leray--Serre spectral sequence for $H^n_G(T^2; \Z)$ degenerates at $E_2$ and the relevant extension problems are trivial, so that
\begin{itemize}\itemsep=0pt
\item[$(a)$] $F^2H^3_G(T^2; \Z) \cong E_2^{2, 1} \oplus E_2^{3, 0}$,

\item[$(b)$] $H^n_G(T^2; \Z) \cong \bigoplus_{p + q = n} E_2^{p, q}$ for all $n \in \Z$.
\end{itemize}
\end{lem}

\begin{proof} The stable splitting in Lemma~\ref{lem:stable_splitting} immediately gives the f\/irst isomorphism. For the trivial $G$-space $\pt$, the Leray--Serre spectral sequence clearly degenerates at~$E_2$, and we have
\begin{gather*}
H^n_G(\pt; \Z) \cong H^n_{\mathrm{group}}\big(G; H^0(\pt; \Z)\big).
\end{gather*}
Since $H^0(\pt; \Z) \cong H^0(T^2; \Z)$ as $G$-modules, we get the following identif\/ication of the $E_2$-term~$E_2^{n, 0}$ of the Leray--Serre spectral sequence for $H^n_G(T^2; \Z)$
\begin{gather*}
E_2^{n, 0} = H^n_{\mathrm{group}}\big(G; H^0\big(T^2; \Z\big)\big) \cong H^n_G(\pt; \Z).
\end{gather*}
For the $G$-space $S^1 \vee S^1$, we can see, as in the proof of Lemma~\ref{lem:1_dim_complex}, that the Leray--Serre spectral sequence also degenerates at $E_2$ and the extension problems are trivial. Because $H^1(S^1 \vee S^1; \Z) \cong H^1(T^2; \Z)$ as $G$-modules, the $E_2$-term $E_2^{n-1, 1}$ of the Leray--Serre spectral sequence for $H^n_G(T^2; \Z)$ is
\begin{gather*}
E_2^{n-1, 1} = H^{n-1}_{\mathrm{group}}\big(G; H^1\big(T^2; \Z\big)\big) \cong \tilde{H}^n_G\big(S^1 \vee S^1; \Z\big).
\end{gather*}
Exactly in the same way, we have
\begin{gather*}
E_2^{n-2, 2} = H^{n-2}_{\mathrm{group}}\big(G; H^2\big(T^2; \Z\big)\big) \cong \tilde{H}^n_G\big(T^2/S^1 \vee S^1; \Z\big),
\end{gather*}
since $H^2(T^2/S^1 \vee S^1; \Z) \cong H^2(T^2; \Z)$ as $G$-modules. The f\/irst isomorphism now gives $H^n_G(T^2; \Z)$ $\cong E_2^{n, 0} \oplus E_2^{n-1, 1} \oplus E_2^{n - 2, 2}$, which also implies the triviality of the spectral sequence.
\end{proof}

\subsection{The outline of computations}

Theorems \ref{thm:main} and~\ref{thm:equiv_coh} follow from case by case computations. As mentioned in Section~\ref{sec:introduction}, three methods are applicable.
\begin{enumerate}\itemsep=0pt
\item In the cases of \textsf{p2} and \textsf{pm/pg}, the point group $\Z_2 = D_1$ acts on the torus \mbox{$T^2 = S^1 \times S^1$} preserving the direct product structure, so that we can think of $T^2$ as a stack of certain equivariant circle bundles ($\Z_2$-equivariant principal circle bundles and/or `Real' circle bundles in the sense of \cite{G}). For such circle bundles, we can use the Gysin exact sequence to compute the equivariant cohomology, as detailed in \cite{G}. In particular, in the cases of \textsf{p2} and \textsf{pm/pg}, the Gysin exact sequences are split, and the computations are very simple. (The computation by using the Gysin sequence is also valid for \textsf{cm}, even though the sequence is non-split.) In the case of \textsf{p2}, we do not need to compute the Leray--Serre spectral sequence, since the third cohomology is trivial. In the case of \textsf{pm/pg}, the spectral sequence can be computed directly.

\item In the cases of \textsf{p4}, \textsf{cm}, \textsf{pmm/pmg/pgg}, \textsf{cmm} and \textsf{p4m/p4g}, we can verify that the action of the point group on $T^2$ satisf\/ies the assumptions of Lemma \ref{lem:stable_splitting}, by inspecting the explicit presentation in Appendix \ref{sec:appendix:wallpapers}. Hence we can apply Lemma \ref{lem:equiv_coh_splitting_case} to the computation of the equivariant cohomology and the spectral sequence. In this application, the only non-trivial part is the equivariant cohomology of the invariant subspace $S^1 \vee S^1$, which we compute by using the Mayer--Vietoris exact sequence.

\item In the cases of \textsf{p3}, \textsf{p6}, \textsf{p3m1}, \textsf{p31m} and \textsf{p6m}, the computation can be divided into two parts. One part is to compute $H^3_P(T^2; \Z)$. This is carried out by taking a $P$-CW decomposition of~$T^2$, and by using the Mayer--Vietoris exact sequence and the exact sequence for a pair. The other part is to compute the Leray--Serre spectral sequence. In this part, we need to know the group cohomology with coef\/f\/icients in $H^1(T^2; \Z)$. For this aim, we take an invariant subspace $Y \subset T^2$ of one dimension. The equivariant cohomology of $Y$ is computed by using Mayer--Vietoris sequence, which allows us to know the group cohomology with its coef\/f\/icients in $H^1(Y; \Z)$ through Lemma \ref{lem:1_dim_complex}. The coef\/f\/icients $H^1(Y; \Z)$ and $H^1(T^2; \Z)$ are related by a short exact sequence. The associated long exact sequence then computes the group cohomology with coef\/f\/icients in $H^1(T^2; \Z)$. Depending on the cases, one of these parts happens to be enough to complete the computation.
\end{enumerate}

In the cases of \textsf{p2}, \textsf{p4m/p4g} and \textsf{p6m}, the detail of the computation is provided in the following subsections. The details for the other cases can be found in old versions of \verb|arXiv:1509.09194|.

\subsection{\textsf{p2}}

The lattice $\Pi \subset \R^2$ is the standard one $\Pi = \Z \oplus \Z$ and the point group $P = \Z_2 = \{ \pm 1 \}$ acts on~$\Pi$ and~$\R^2$ by $(x, y) \mapsto (-x, -y)$.

\begin{thm}[\textsf{p2}] \label{thm:p2}
The $\Z_2$-equivariant cohomology of $T^2$ is given as follows
$$
\begin{array}{|c|c|c|c|c|}
\hline
& n = 0 & n = 1 & n = 2 & n = 3 \\
\hline
H^n_{\Z_2}(T^2; \Z) & \Z & 0 & \Z \oplus \Z_2^{\oplus 3} & 0 \tsep{2pt}\bsep{2pt}\\
\hline
\end{array}
$$
\end{thm}

\begin{proof}
We use the Gysin exact sequence for `Real' circle bundles in \cite{G}: We write $H^n_{\Z_2}(X) = H^n_{\Z_2}(X; \Z)$ for the equivariant cohomology and $H^n_{\pm}(X) \cong H^n_{\Z}(X; \Z(1))$ for a variant of the equivariant cohomology, which can be formulated by the equivariant cohomology with local coef\/f\/icients. The torus $T^2$ is the product of two copies of $\tilde{S}^1$, where $\tilde{S}^1 = U(1)$ is the circle with the involution $z \mapsto z^{-1}$. We can think of $\tilde{S}^1 \times \tilde{S}^1$ as the trivial `Real' circle bundle on $\tilde{S}^1$. Similarly, $\tilde{S}^1$ is the trivial `Real' circle bundle on $\pt$. The Gysin exact sequences for these `Real' circle bundles are split, and we f\/ind
\begin{gather*}
H^n_{\Z_2}\big(T^2\big)
\cong H^n_{\Z_2}\big(\tilde{S}^1\big) \oplus H^{n-1}_{\pm}\big(\tilde{S}^1\big) \cong H^n_{\Z_2}(\pt) \oplus H^{n-1}_{\pm}(\pt)
\oplus H^{n-1}_{\pm}(\pt) \oplus H^{n-2}_{\Z_2}(\pt).
\end{gather*}
As given in~\cite{G}, the cohomology $H^n_{\pm}(\pt)$ is isomorphic to $\Z_2$ if $n > 0$ is odd, and is trivial otherwise. We already know $H^n_{\Z_2}(\pt)$, and get $H^n_{\Z_2}(T^2)$ easily.
\end{proof}

\subsection{\textsf{p4m/p4g}}\label{subsec:p4m_p4g}

The lattice $\Pi = \Z^2 \subset \R^2$ is standard. The point group is
$P =D_4 =\langle C_4, \sigma_x \,|\, C_4^4, \sigma_x^2, \sigma_xC_4\sigma_xC_4 \rangle$. The $D_4$-action on $\Pi$ and $\R^2$ is given by the following matrix presentation:
\begin{gather*}
C_4 = \left(
\begin{matrix}
0 & -1\\
1 & 0
\end{matrix}
\right), \qquad \sigma_x =
\left(
\begin{matrix}
-1 & 0\\
0 & 1
\end{matrix}
\right).
\end{gather*}
In the rest of this subsection, we will use the following notations to indicate elements in $D_4$:
\begin{alignat*}{5}
&1, \qquad &&C_4, \qquad && C_2 = C_4^2, \qquad && C_4^{-1} = C_4^3,& \\
&\sigma_x, \qquad && \sigma_d = \sigma_x C_4,\qquad && \sigma_y = C_2\sigma_x, \qquad && \sigma_d' = C_4 \sigma_x.&
\end{alignat*}
The closure of a fundamental domain is $\{ s (1, 0) + t (0, 1) \in \R^2 \,|\, 0 \le s, t \le 1 \}$. Then we f\/ind that the $D_4$-action on $T^2 = \R^2/\Pi$ satisf\/ies the assumptions in Lemma \ref{lem:stable_splitting}, in which $\pt = (0, 0)$ and
\begin{gather*}
S^1 \vee S^1 \cong Y = ((\R \oplus 0)/(\Z \oplus 0)) \vee ((0 \oplus \R)/(0 \oplus \Z)).
\end{gather*}
To apply Lemma \ref{lem:equiv_coh_splitting_case}, we compute the cohomology of $Y$:

\begin{lem}
The equivariant cohomology of $Y$ is as follows:
$$
\begin{array}{|c|c|c|c|c|}
\hline
& n = 0 & n = 1 & n = 2 & n = 3 \\
\hline
H^n_{D_4}(Y; \Z) &
\Z & 0 & \Z_2^{\oplus 3} & \Z_2^{\oplus 2} \tsep{2pt}\bsep{2pt}\\
\hline
\end{array}
$$
\end{lem}

\begin{proof}
We use the Mayer--Vietoris exact sequence: Cover $Y$ by invariant subspaces $U$ and $V$ with the following $D_4$-equivariant homotopy equivalences
\begin{gather*}
U \simeq \pt, \qquad V \simeq D_4/D_2^{(v)}, \qquad U \cap V \simeq D_4/\Z_2^{(v)},
\end{gather*}
where $D_2^{(v)} = \{ 1, C_2, \sigma_x, \sigma_y \} \cong D_2$ and $\Z_2^{(v)} = \{ 1, \sigma_y \} \cong \Z_2$. We can summarize the equivariant cohomology of these spaces in low degrees as follows:
$$
\begin{array}{|c|c|c|c|}
\hline
n = 3 & & \Z_2 \oplus \Z_2 & 0 \\
\hline
n = 2 & & \Z_2^{\oplus 2} \oplus \Z_2^{\oplus 2} \tsep{2pt}\bsep{2pt} & \Z_2 \\
\hline
n = 1 & & 0 & 0 \\
\hline
n = 0 & & \Z \oplus \Z & \Z \\
\hline
& H^n_{D_4}(Y) & H^n_{D_4}(U) \oplus H^n_{D_4}(V) &
H^n_{D_4}(U \cap V)\tsep{2pt}\bsep{2pt}\\
\hline
\end{array}
$$
In the Mayer--Vietoris exact sequence
\begin{gather*}
\cdots \to H^n_{D_4}(Y) \to H^n_{D_4}(U) \oplus H^n_{D_4}(V) \overset{\Delta}{\to} H^n_{D_4}(U \cap V) \to H^{n+1}_{D_4}(Y) \to \cdots,
\end{gather*}
the map $\Delta \colon H^n_{D_4}(U) \oplus H^n_{D_4}(V) \to H^n_{D_4}(U \cap V)$ is expressed as $\Delta(u, v) = j_U^*(u) - j_V^*(v)$, where $j_U \colon U \cap V \to U$ and $j_V \colon U \cap V \to V$ are the inclusions. Under the natural identif\/ications
\begin{gather*}
\begin{split}
& H^2_{D_4}(V) \cong H^2_{D_2^{(v)}}(\pt) \cong \operatorname{Hom}\big(D_2^{(v)}, U(1)\big), \\
& H^2_{D_4}(U \cap V) \cong H^2_{\Z_2^{(v)}}(\pt) \cong \operatorname{Hom}\big(\Z_2^{(v)}, U(1)\big),
\end{split}
\end{gather*}
the map $j_U^*$ agrees with the homomorphism induced from the inclusion $\Z_2^{(v)} \to D_2^{(v)}$. This implies that $j_U^*$ is surjective, and so is $\Delta$ in degree~$2$. Clearly, $\Delta \colon H^0_{D_4}(U) \oplus H^0_{D_4}(V) \to H^0_{D_4}(U \cap V)$ is identif\/ied with the homomorphism $\Z \oplus \Z \to \Z$ given by $(m, n) \mapsto m - n$. Hence we can solve the Mayer--Vietoris exact sequence for $\{ U, V \}$ to get the result claimed in this lemma.
\end{proof}

\begin{thm}[\textsf{p4m/p4g}]
The $D_4$-equivariant cohomology of $T^2$ in low degrees is as follows:
$$
\begin{array}{|c|c|c|c|c|}
\hline
 & n = 0 & n = 1 & n = 2 & n = 3 \\
\hline
H^n_{D_4}(T^2; \Z) & \Z & 0 & \Z_2^{\oplus 3} & \Z_2^{\oplus 3} \tsep{2pt}\bsep{2pt} \\
\hline
\end{array}
$$
We also have $F^2H^3_{D_4}(T^2; \Z) \cong \Z_2^{\oplus 2}$.
\end{thm}

\begin{proof}
In the Leray--Serre spectral sequence
$E_2^{p, q} = H^p_{\mathrm{group}}(D_4; H^q(T^2; \Z))$, the $D_4$-modules $H^0(T^2)$, $H^1(T^2)$ and $H^2(T^2)$ are identif\/ied with the trivial $D_4$-module $\Z$, $H^1(Y)$ and the $D_4$-module $\tilde{\Z}$ in Lemma~\ref{lem:group_coh_orientation_coeff}, respectively. Using Lemmas~\ref{lem:1_dim_complex} and~\ref{lem:group_coh_orientation_coeff}, we can summarize the $E_2$-terms as follows:
$$
\begin{array}{|c|c|c|c|c|}
\hline
q = 3 & 0 & 0 & 0 & 0 \\
\hline
q = 2 & 0 & \Z_2 & & \\
\hline
q = 1 & 0 & \Z_2 & \Z_2 & \\
\hline
q = 0 & \Z & 0 & \Z_2^{\oplus 2} & \Z_2\tsep{2pt}\bsep{2pt} \\
\hline
E_2^{p, q} & p = 0 & p = 1 & p = 2 & p = 3 \tsep{2pt}\bsep{2pt}\\
\hline
\end{array}
$$
Now the proof is completed by Lemma \ref{lem:equiv_coh_splitting_case}.
\end{proof}

\subsection{\textsf{p6m}}\label{subsec:p6m}

We let $\Pi = \Z a \oplus \Z b \subset \R^2$ be the lattice spanned by
$a = \left( \begin{matrix} 1 \\ 0 \end{matrix} \right)$ and
$b = \left( \begin{matrix} 1/2 \\ \sqrt{3}/2 \end{matrix} \right)$.
The point group $P$ is
$D_6 =\langle C, \sigma_1 \,|\, C^6, \sigma_1^2, \sigma_1C\sigma_1C \rangle=
\{1, C, C^2, C^3, C^4, C^5,\sigma_1, \sigma_2, \sigma_3, \sigma_4, \sigma_5, \sigma_6\}$, where $\sigma_\ell = C^{\ell - 1}\sigma_1$. This group acts on $\Pi$ and $\R^2$ through the inclusion $D_6 \subset {\rm O}(2)$ def\/ined by
\begin{gather*}
C = \left(
\begin{matrix}
1/2 & -\sqrt{3}/2 \\
\sqrt{3}/2 & 1/2
\end{matrix}
\right), \qquad \sigma_1 =
\left(
\begin{matrix}
1 & 0 \\
0 & -1
\end{matrix}
\right).
\end{gather*}
If we use the identif\/ications $a = 1$ and $b = \tau = \exp 2\pi i/6$ under $\R^2 = \C$, then the actions of $C \in D_6$ and $\sigma_1$ are given by the multiplication by $\tau$ and the complex conjugation, respectively. The closure of a fundamental domain is $\{ s a + t b \,|\, 0 \le s, t \le 1 \}$ or equivalently $\{ s + t \tau \,|\, 0 \le s, t \le 1 \}$. We decompose this region to def\/ine a $D_6$-CW decomposition of $T^2$ as follows:
$$
\begin{array}{|c|c|c|}
\hline
\mbox{0-cell} & \mbox{1-cell} & \mbox{2-cell} \\
\hline
\tilde{e}^0_0 = \pt &
\tilde{e}^1_{01} = (D_6/\{ 1, \sigma_1 \}) \times e^1 &
\tilde{e}^2 = D_6 \times e^2 \tsep{2pt}\\
\tilde{e}^0_1 = D^6/\{ 1, C^3, \sigma_1, \sigma_4 \} &
\tilde{e}^1_{02} = (D_6/\{ 1, \sigma_2 \}) \times e^1 & \\
\tilde{e}^0_2 = D^6/\{ 1, C^2, C^4, \sigma_2, \sigma_4, \sigma_6 \} &
\tilde{e}^1_{12} = (D_6/\{ 1, \sigma_4 \}) \times e^1 & \bsep{2pt}\\
\hline
\end{array}
$$

\begin{itemize}\itemsep=0pt
\item
($0$-cell)
The $0$-cell $\tilde{e}^0_0 = (D_6/D_6) \times e^0 = \pt$ is the unique f\/ixed point on $T^2$. The other $0$-cells are def\/ined as follows:
\begin{gather*}
\tilde{e}^0_1 =\left\{\frac{1}{2}, \frac{\tau}{2}, \frac{1+\tau}{2}\right\}\cong \big(D_6/\big\{ 1, C^3, \sigma_1, \sigma_4 \big\}\big) \times e^0, \\
\tilde{e}^0_2 =\left\{\frac{1 + \tau}{3}, \frac{2(1+\tau)}{3}\right\}\cong \big(D_6/\big\{ 1, C^2, C^4, \sigma_2, \sigma_4, \sigma_6 \big\}\big) \times e^0.
\\
\xygraph{
!{<0cm,0cm>;<1cm,0cm>:<0cm,1cm>::}
!{(0,0)}*+{\bullet}="a"
!{(2,0)}="b"
!{(1,1.7)}="c"
!{(3,1.7)}="d"
!{(1,0.6)}="e"
!{(2,1.2)}="f"
"a"-@{.}"b"
"a"-@{.}"c"
"b"-@{.}"c"
"b"-@{.}"d"
"c"-@{.}"d"
}
\qquad
\xygraph{
!{<0cm,0cm>;<1cm,0cm>:<0cm,1cm>::}
!{(0,0)}="a"
!{(2,0)}="b"
!{(1,1.7)}="c"
!{(3,1.7)}="d"
!{(1,0)}*+{\bullet}="g"
!{(0.5,0.85)}*+{\bullet}="h"
!{(1.5,0.85)}*+{\bullet}="k"
"a"-@{.}"g"
"g"-@{.}"b"
"a"-@{.}"h"
"h"-@{.}"c"
"b"-@{.}"k"
"k"-@{.}"c"
"b"-@{.}"d"
"c"-@{.}"d"
}
\qquad
\xygraph{
!{<0cm,0cm>;<1cm,0cm>:<0cm,1cm>::}
!{(0,0)}="a"
!{(2,0)}="b"
!{(1,1.7)}="c"
!{(3,1.7)}="d"
!{(1,0.6)}*+{\bullet}="e"
!{(2,1.2)}*+{\bullet}="f"
"a"-@{.}"b"
"a"-@{.}"c"
"b"-@{.}"c"
"b"-@{.}"d"
"c"-@{.}"d"
}
\end{gather*}

\item
($1$-cell)
For $0 \le i < j \le 2$, the $1$-cell $\tilde{e}^1_{ij}$ consists of the six segments connecting $\tilde{e}^0_i$ and $\tilde{e}^0_j$. They are of the forms $\tilde{e}^1_{01} = (D_6/\{ 1, \sigma_1 \}) \times e^1$, $\tilde{e}^1_{02} = (D_6/\{ 1, \sigma_2 \}) \times e^1$ and $\tilde{e}^1_{12} = (D_6/\{ 1, \sigma_4 \}) \times e^1$.
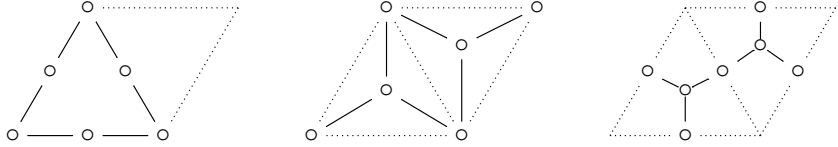
\begin{gather*}
\xygraph{
!{<0cm,0cm>;<1cm,0cm>:<0cm,1cm>::}
!{(0,0)}*+{\circ}="a"
!{(2,0)}*+{\circ}="b"
!{(1,1.7)}*+{\circ}="c"
!{(3,1.7)}="d"
!{(1,0)}*+{\circ}="g"
!{(0.5,0.85)}*+{\circ}="h"
!{(1.5,0.85)}*+{\circ}="k"
"a"-"g"
"g"-"b"
"a"-"h"
"h"-"c"
"b"-"k"
"k"-"c"
"b"-@{.}"d"
"c"-@{.}"d"
}
\qquad
\xygraph{
!{<0cm,0cm>;<1cm,0cm>:<0cm,1cm>::}
!{(0,0)}*+{\circ}="a"
!{(2,0)}*+{\circ}="b"
!{(1,1.7)}*+{\circ}="c"
!{(3,1.7)}*+{\circ}="d"
!{(1,0.6)}*+{\circ}="e"
!{(2,1.2)}*+{\circ}="f"
"a"-@{.}"b"
"a"-@{.}"c"
"b"-@{.}"c"
"b"-@{.}"d"
"c"-@{.}"d"
"a"-"e"
"b"-"e"
"c"-"e"
"d"-"f"
"b"-"f"
"c"-"f"
}
\qquad
\xygraph{
!{<0cm,0cm>;<1cm,0cm>:<0cm,1cm>::}
!{(0,0)}="a"
!{(2,0)}="b"
!{(1,1.7)}="c"
!{(3,1.7)}="d"
!{(1,0.6)}*{\circ}="e"
!{(2,1.2)}*{\circ}="f"
!{(1,0)}*+{\circ}="g"
!{(0.5,0.85)}*+{\circ}="h"
!{(2,1.7)}*+{\circ}="i"
!{(2.5,0.85)}*+{\circ}="j"
!{(1.5,0.85)}*+{\circ}="k"
"e"-"g"
"e"-"h"
"e"-"k"
"f"-"k"
"f"-"i"
"f"-"j"
"a"-@{.}"g"
"g"-@{.}"b"
"a"-@{.}"h"
"h"-@{.}"c"
"b"-@{.}"k"
"k"-@{.}"c"
"b"-@{.}"j"
"j"-@{.}"d"
"c"-@{.}"i"
"i"-@{.}"d"
}
\end{gather*}

\item
($2$-cell)
The $2$-cell $\tilde{e}^2 = D_6 \times e^2$ consists of the twelve small triangular regions surrounded by the $1$-cells.
\end{itemize}

Let $Y \subset T^2$ be the invariant subspace $Y = \tilde{e}^0_0 \cup \tilde{e}^0_1 \cup \tilde{e}^1_{01}$.

\begin{lem} \label{lem:p6m_Y}
The equivariant cohomology of $Y$ is given as follows:
$$
\begin{array}{|c|c|c|c|c|}
\hline
 & n = 0 & n = 1 & n = 2 & n = 3 \\
\hline
H^n_{D_6}(Y; \Z) & \Z & 0 & \Z_2^{\oplus 3} & \Z_2^{\oplus 2} \tsep{2pt}\bsep{2pt} \\
\hline
\end{array}
$$
\end{lem}

\begin{proof}
We can f\/ind $D_6$-invariant subspaces $U$ and $V$ in $Y$ which have the following equivariant homotopy equivalences
\begin{gather*}
U \simeq \tilde{e}^0_0 = \pt, \qquad
V \simeq \tilde{e}^0_1 = D_6/D_2, \qquad
U \cap V \simeq \tilde{e}^1_{01} \simeq D_6/\Z_2^{(1)},
\end{gather*}
where $D_2 = \{ 1, C^3, \sigma_1, \sigma_4 \}$ and $\Z_2^{(1)} = \{ 1, \sigma_1 \}$. The equivariant cohomology groups of these spaces can be summarized as follows:
$$
\begin{array}{|c|c|c|c|}
\hline
n = 3 & & \Z_2 \oplus \Z_2 & 0 \\
\hline
n = 2 & & \Z_2^{\oplus 2} \oplus \Z_2^{\oplus 2} & \Z_2 \tsep{2pt}\bsep{2pt}\\
\hline
n = 1 & & 0 & 0 \\
\hline
n = 0 & & \Z \oplus \Z & \Z \\
\hline
& H^n_{D_6}(Y) & H^n_{D_6}(U) \oplus H^n_{D_6}(V) & H^n_{D_6}(U \cap V)\tsep{2pt}\bsep{2pt}\\
\hline
\end{array}
$$
In the Mayer--Vietoris exact sequence
\begin{gather*}
\cdots \to H^n_{D_6}(Y) \to H^n_{D_6}(U) \oplus H^n_{D_6}(V) \overset{\Delta}{\to} H^n_{D_6}(U \cap V) \to H^{n+1}_{D_6}(Y) \to \cdots,
\end{gather*}
the homomorphism $\Delta$ is expressed as $\Delta (u, v) = j_U^*(u) - j_V^*(v)$ with $j_U \colon U \cap V \to U$ and $j_V \colon U \cap V \to V$ the inclusions. This immediately determines $H^0_{D_6}(Y) \cong \Z$ and $H^1_{D_6}(Y) = 0$. To complete the proof, we recall the identif\/ications
\begin{gather*}
H^2_{D_6}(U) \cong \operatorname{Hom}(D_6, U(1)) \cong \Z_2^{\oplus 2}, \qquad
H^2_{D_6}(V) \cong \operatorname{Hom}(D_2, U(1)) \cong \Z_2^{\oplus 2}, \\
H^2_{D_6}(U \cap V) \cong \operatorname{Hom}\big(\Z_2^{(1)}, U(1)\big) \cong \Z_2,
\end{gather*}
under which $j_U^*$ and $j_V^*$ are induced from the inclusions $D_2 \to D_6$ and $\Z_2^{(1)} \to D_6$. As a basis of~$H^2_{D_6}(U)$ we can choose the following $1$-dimensional representations $\rho_i\colon D_6 \to U(1)$ of $D_6$
\begin{gather*}
\rho_1 \colon \
\begin{cases}
C \mapsto 1, \\
\sigma_1 \mapsto -1,
\end{cases}
\qquad
\rho_2 \colon \
\begin{cases}
C \mapsto -1, \\
\sigma_1 \mapsto 1.
\end{cases}
\end{gather*}
Similarly, we can choose the following $1$-dimensional representations $\rho'_i$ of $D_2 = \{ 1, \sigma_1, C^3, \sigma_4 \}$ as a basis of $H^2_{D_6}(V)$
\begin{gather*}
\rho'_1 \colon \
\begin{cases}
C^3 \mapsto 1, \\
\sigma_1 \mapsto -1,
\end{cases}
\qquad
\rho'_2 \colon \
\begin{cases}
C^3 \mapsto -1, \\
\sigma_1 \mapsto 1.
\end{cases}
\end{gather*}
Now, we can see $H^2_{D_6}(Y) \cong \operatorname{Ker}\Delta \cong \Z_2^{\oplus 3}$, and it has the following basis
\begin{gather*}
\{ (\rho_1, \rho'_1), (\rho_2, \rho_2'), (0, \rho'_2) \} \subset \operatorname{Hom}(D_6, U(1)) \oplus \operatorname{Hom}(D_4, U(1)).
\end{gather*}
We can also see that $\Delta$ is surjective, and $H^3_{D_6}(Y) \cong \Z_2^{\oplus 2}$.
\end{proof}

Let $X_1$ be the $1$-skeleton of the $D_6$-CW complex $T^2$.

\begin{lem}
$H^3_{D_6}(X_1; \Z) \cong \Z_2^{\oplus 2}$.
\end{lem}

\begin{proof}
We cover $X_1 = \tilde{e}^0_0 \cup \tilde{e}^0_1 \cup \tilde{e}^0_2 \cup \tilde{e}^1_{01} \cup \tilde{e}^1_{02} \cup \tilde{e}^1_{12}$ by invariant subspaces $U'$ and $V'$ which admit the following equivariant homotopy equivalences
\begin{gather*}
U' \simeq Y, \qquad
V' \simeq \tilde{e}^0_2 = D_6/D_3, \qquad
U' \cap V' \simeq \tilde{e}^1_{02} \sqcup \tilde{e}^1_{12} \simeq D_6/\Z_2^{(2)} \sqcup D_6/\Z_2^{(4)},
\end{gather*}
where $D_3 = \{ 1, C^2, C^4, \sigma_2, \sigma_4, \sigma_6 \}$, $\Z_2^{(2)} = \{ 1, \sigma_2 \}$ and $\Z_2^{(4)} = \{ 1, \sigma_4 \}$. The equivariant cohomology groups of these spaces are summarized as follows:
$$
\begin{array}{|c|c|c|c|}
\hline
n = 3 & & \Z_2^{\oplus 2} \oplus 0 & 0\tsep{2pt}\bsep{2pt} \\
\hline
n = 2 & & \Z_2^{\oplus 3} \oplus \Z_2 & \Z_2 \oplus \Z_2 \tsep{2pt}\bsep{2pt}\\
\hline
n = 1 & & 0 & 0 \\
\hline
n = 0 & & \Z \oplus \Z & \Z \oplus \Z \\
\hline
& H^n_{D_6}(X_1) & H^n_{D_6}(U') \oplus H^n_{D_6}(V') & H^n_{D_6}(U' \cap V') \tsep{2pt}\bsep{2pt}\\
\hline
\end{array}
$$
The homomorphism $\Delta$ in the Mayer--Vietoris exact sequence
\begin{gather*}
\cdots \to H^2_{D_6}(X_1) \to H^2_{D_6}(U') \oplus H^2_{D_6}(V') \overset{\Delta}{\to} H^2_{D_6}(U' \cap V') \to H^3_{D_6}(X_1) \to \cdots
\end{gather*}
is expressed as $\Delta(u, v) = j_{U'}^*(u) - j_{V'}^*(v)$ by using the inclusions $j_{U'} \colon U' \cap V' \to U'$ and $j_{V'} \colon U' \cap V' \to V'$. An inspection proves that $j_{U'}^*$ agrees with the composition of the following two homomorphisms:
\begin{itemize}\itemsep=0pt
\item[(i)]
the inclusion that follows from the calculation of $H^2_{D_6}(Y)$ in Lemma \ref{lem:p6m_Y}
\begin{gather*}
H^2_{D_6}(U') \cong H^2_{D_6}(Y) \longrightarrow
\operatorname{Hom}(D_6, U(1)) \oplus \operatorname{Hom}(D_2, U(1)).
\end{gather*}

\item[(ii)]
the direct sum $i_2^* \oplus i_4^*$ of the homomorphisms
\begin{gather*}
i^*_2 \colon \ \operatorname{Hom}(D_6, U(1)) \to \operatorname{Hom}\big(\Z_2^{(2)}, U(1)\big), \qquad
i^*_4 \colon \ \operatorname{Hom}(D_2, U(1)) \to \operatorname{Hom}\big(\Z_2^{(4)}, U(1)\big),
\end{gather*}
induced from the inclusions $i_2\colon \Z_2^{(2)} \to D_6$ and $\Z_2^{(4)} \to D_2$.

\end{itemize}
Then, using the basis presented in the calculation of $H^2_{D_6}(Y)$, we f\/ind
\begin{gather*}
j_{U'}^*(\rho_1, \rho_1') = (\rho, \rho), \qquad j_{U'}^*(\rho_2, \rho_2') = (\rho, \rho), \qquad
j_{U'}^*(0, \rho_2') = (0, \rho),
\end{gather*}
where $\rho\colon \Z_2 \to \Z_2$ is the identity map generating $\operatorname{Hom}(\Z_2, U(1)) \cong \Z_2$. Hence $j_{U'}^*$ as well as $\Delta$ are surjective, and $H^3_{D_6}(X_1) \cong H^3_{D_6}(Y) \cong \Z_2^{\oplus 2}$.
\end{proof}

\begin{thm}[\textsf{p6m}]
$H^3_{D_6}(T^2; \Z) \cong \Z_2^{\oplus 2}$.
\end{thm}

\begin{proof}
The relevant part of the exact sequence for the pair $(T^2, X_1)$ is
\begin{gather*}
H^3_{D_6}\big(T^2, X_1\big) \to H^3_{D_6}\big(T^2\big) \to H^3_{D_6}(X_1) \to H^4_{D_6}\big(T^2, X_1\big).
\end{gather*}
By means of the excision axiom, we have $H^n_{D_6}(T^2, X_1) \cong H^{n-2}(\pt)$. Therefore we get $H^3_{D_6}(T^2) \cong H^3_{D_6}(X_1) \cong \Z_2^{\oplus 2}$.
\end{proof}

Let $\hat{\Z} = \Z_{\phi_1}$ be the $D_6$-module such that its underlying group is $\Z$ and $D_6$ acts via the homomorphism $\phi_1\colon D_6 \to \Z_2$ given by $\phi_1(C) = -1$ and $\phi_1(\sigma_1) = 1$.

\begin{lem} \label{lem:p6m_short_exact_sequence_of_modules}
There is an exact sequence of $D_6$-modules
\begin{gather*}
0 \to H^1\big(T^2; \Z\big) \to H^1(Y; \Z) \overset{\pi}{\to} \hat{\Z} \to 0
\end{gather*}
admitting a module homomorphism $s\colon \hat{\Z} \to H^1(Y; \Z)$ such that $\pi \circ s = 3$.
\end{lem}

\begin{proof}
Let $\eta_1, \eta_2 \in H_1(T^2)$ be the homology classes of the loops going along the vectors $1$ and $\tau$ respectively in the fundamental domain, which form a basis of $H_1(T^2) \cong \Z^2$. Also, let $\gamma_1, \gamma_2, \gamma_3 \in H_1(Y)$ be the homology classes of loops along $1$, $\tau$ and $\tau - 1$, which form a basis of $H_1(Y) \cong \Z^3$. The inclusion map $i\colon Y \to T^2$ relates these bases by $i_*\gamma_1 = \eta_1$, $i_*\gamma_2 = \eta_2$ and $i_*\gamma_3 = \eta_2 - \eta_1$. The actions of $C \in D_6$ and $\sigma_1$ on these bases are
\begin{gather*}
\begin{cases}
C_*\eta_1 = \eta_2, \\
C_*\eta_2 = \eta_2 - \eta_1,
\end{cases}
\qquad
\begin{cases}
{\sigma_1}_*\eta_1 = \eta_1, \\
{\sigma_1}_*\eta_2 = \eta_1 - \eta_2,
\end{cases}
\qquad
\begin{cases}
C_*\gamma_1 = \gamma_2, \\
C_*\gamma_2 = \gamma_3, \\
C_*\gamma_3 = - \gamma_1,
\end{cases}
\qquad
\begin{cases}
{\sigma_1}_*\gamma_1 = \gamma_1, \\
{\sigma_1}_*\gamma_2 = - \gamma_3, \\
{\sigma_1}_*\gamma_3 = - \gamma_2.
\end{cases}\!\!\!
\end{gather*}
Let $\{ h_1, h_2 \} \subset H^1(T^2)$ and $\{ g_1, g_2, g_3 \} \subset H^1(Y)$ be dual to the homology bases. They are related by $i^*h_1 = g_1 - g_3$ and $i^*h_2 = g_2 + g_3$, and the induced $D_6$-actions are as follows.
\begin{gather*}
\begin{cases}
C^*h_1 = - h_2, \\
C^*h_2 = h_1 + h_2,
\end{cases}\qquad
\begin{cases}
\sigma_1^*h_1 = h_1 + h_2, \\
\sigma_1^*h_2 = - h_2,
\end{cases}\qquad
\begin{cases}
C^*g_1 = - g_3, \\
C^*g_2 = g_1, \\
C^*g_3 = g_2,
\end{cases}\qquad
\begin{cases}
\sigma_1^*g_1 = g_1, \\
\sigma_1^*g_2 = - g_3, \\
\sigma_1^*g_3 = - g_2.
\end{cases}
\end{gather*}
These expressions allow us to prove that the cokernel of the homomorphism $i^* \colon H^1(T^2) \to H^2(Y)$ is isomorphic to $\hat{\Z}$, yielding the exact sequence. The homomorphism $s\colon \hat{\Z} \to H^1(Y)$ is given by $s(1) = g_1 - g_2 + g_3$.
\end{proof}

\begin{lem}
$H^n_{\mathrm{group}}(D_6; H^1(T^2; \Z)) = 0$ for $n = 0, 1, 2$.
\end{lem}

\begin{proof}
We use the long exact sequence in group cohomology induced from the exact sequence $0 \to H^1(T^2) \to H^1(Y) \overset{\pi}{\to} \hat{\Z} \to 0$ in coef\/f\/icients. By Lemmas~\ref{lem:1_dim_complex}, \ref{lem:p6m_Y} and~\ref{lem:cohomology_of_point_twisted_case} to be given in Section~\ref{sec:twisted_case}, we get the following:
$$
\begin{array}{|c|c|c|c|}
\hline
n = 2 & & \Z_2 & \Z_2 \\
\hline
n = 1 & & \Z_2 & \Z_2 \\
\hline
n = 0 & & 0 & 0 \\
\hline
& H^n_{\mathrm{group}}(D_6; H^1(T^2)) & H^n_{\mathrm{group}}(D_6; H^1(Y)) &
H^n_{\mathrm{group}}(D_6; \hat{\Z})\tsep{2pt}\bsep{2pt}\\
\hline
\end{array}
$$
It is clear that $H^0_{\mathrm{group}}(D_6; H^1(T^2)) = 0$. The homomorphism in group cohomology induced from $\pi\colon H^1(Y) \to \hat{\Z}$ is surjective in degree $1$ and $2$, because $\pi \circ s = 3$. This leads to the remaining vanishing.
\end{proof}

\begin{thm}[\textsf{p6m}]
The following holds true:
\begin{itemize}\itemsep=0pt
\item[$(a)$] $F^2H^3_{D_6}(T^2; \Z) \cong \Z_2$,

\item[$(b)$] the $D_6$-equivariant cohomology of $T^2$ in low degrees is as follows:
\begin{gather*}
H^0_{D_6}\big(T^2; \Z\big) \cong \Z, \qquad H^1_{D_6}\big(T^2; \Z\big) = 0, \qquad H^2_{D_6}\big(T^2; \Z\big) \cong \Z_2^{\oplus 2}.
\end{gather*}
\end{itemize}
\end{thm}

\begin{proof} In the $E_2$-term of the Leray--Serre spectral sequence $E_2^{p, q} = H^p_{\mathrm{group}}(D_6; H^q(T^2; \Z))$, the coef\/f\/icient $H^0(T^2)$ is identif\/ied with the trivial $D_6$-module $\Z$, and $H^2(T^2)$ with $\tilde{\Z}$ as in Lemma~\ref{lem:group_coh_orientation_coeff}. The group cohomology with coef\/f\/icients in $H^1(T^2)$ is already computed, and that in $\tilde{\Z}$ is also computed in Lemma~\ref{lem:group_coh_orientation_coeff}. The $E_2$-terms are summarized as follows:
$$
\begin{array}{|c|c|c|c|c|}
\hline
q = 3 & 0 & 0 & 0 & 0 \\
\hline
q = 2 & 0 & \Z_2 & & \\
\hline
q = 1 & 0 & 0 & 0 & \\
\hline
q = 0 & \Z & 0 & \Z_2^{\oplus 2} & \Z_2 \tsep{2pt}\bsep{2pt}\\
\hline
E_2^{p, q} & p = 0 & p = 1 & p = 2 & p = 3 \tsep{2pt}\bsep{2pt}\\
\hline
\end{array}
$$
This list and Lemma \ref{lem:typical_argument_non_orientation_preserving_case} lead to the theorem.
\end{proof}

\subsection{The proof of Corollary \ref{cor:main}}

The only non-trivial point in the corollary is (c), which we prove here. Let $P$ be the point group of one of the $2$-dimensional space groups. We can assume that $P$ does not preserve the orientation of $T^2$. Then we have
\begin{gather*}
F^2H^3_P\big(T^2; \Z\big) \cong E^{2,1}_2 \oplus E^{3, 0}_2
\end{gather*}
by Lemma \ref{lem:typical_argument_non_orientation_preserving_case}, in which the direct summands are
\begin{gather*}
E^{2, 1}_2 = H^2_{\mathrm{group}}\big(P; H^1\big(T^2; \Z\big)\big), \qquad
E^{3, 0}_2 = H^3_{\mathrm{group}}(P; \Z) \cong H^2_{\mathrm{group}}(P; U(1)).
\end{gather*}
Thus, it suf\/f\/ices to prove that the group cocycles $\tau$ induced from the nonsymmorphic $2$-dimensional space groups as in Section~\ref{sec:quantum_system_to_K} generate $E^{2, 1}_2$.

Recall from Section \ref{sec:quantum_system_to_K} that the group $2$-cocycle $\nu \in Z^2_{\mathrm{group}}(P; \Pi)$ measures the failure for a~space group~$S$ to be a semi-direct product of its point group $P$ and the lattice~$\Pi$, where $\Pi$ is regarded as a left $P$-module naturally. In other words, $S$ is nonsymmorphic if and only if $[\nu] \in H^2_{\mathrm{group}}(P; \Pi)$ is non-trivial.

\begin{lem} \label{lem:cocycle_correspondence}
Let $\Pi$ and $P$ be the lattice and the point group of a $d$-dimensional space group~$S$. Then there is an isomorphism of groups
\begin{gather*}
H^2_{\mathrm{group}}(P; \Pi) \cong H^2_{\mathrm{group}}\big(P; H^1\big(\hat{\Pi}; \Z\big)\big).
\end{gather*}
In particular, this factors through the homomorphisms
\begin{gather*}
H^2_{\mathrm{group}}(P; \Pi) \longrightarrow H^2_{\mathrm{group}}\big(P; C\big(\hat{\Pi}, U(1)\big)\big)
\end{gather*}
given by the assignment of the cocycles $\nu \mapsto \tau$ in Section~{\rm \ref{sec:quantum_system_to_K}} and
\begin{gather*}
H^2_{\mathrm{group}}\big(P; C\big(\hat{\Pi}, U(1)\big)\big) \longrightarrow H^2_{\mathrm{group}}\big(P; H^1\big(\hat{\Pi}; \Z\big)\big)
\end{gather*}
induced from the natural surjection $\delta \colon C(\hat{\Pi}, U(1)) \to H^1(\hat{\Pi}; \Z)$.
\end{lem}

\begin{proof}
Instead of the left $P$-action on the Pontryagin dual $\hat{\Pi} = \operatorname{Hom}(\Pi, U(1))$ def\/ined in Section \ref{sec:quantum_system_to_K}, we consider the natural right action $\hat{k}(m) \mapsto \hat{k}(pm)$ of $p \in P$ on $\hat{k} \in \hat{\Pi}$, from which the left action originates. This choice of the actions does not af\/fect the group cohomology. The right $P$-action on $\hat{\Pi}$ induces by pull-back a left $P$-action on the cohomology $H^1(\hat{\Pi}; \Z)$. Thus, the isomorphism of the group cohomologies will be established once we see $H^1(\hat{\Pi}; \Z) \cong \Pi$ as left $P$-modules. In general, for each element $m \in \Pi \subset \R^d = V$, the path $[0, 1] \to V$, ($t \mapsto tm$) def\/ines a loop in $V/\Pi$. This induces an isomorphism of left $P$-modules $\Pi \cong H_1(V/\Pi; \Z)$. By the universal coef\/f\/icient theorem, the dual $\Pi^* = \operatorname{Hom}(\Pi, \Z)$ of $\Pi$ is identif\/ied with the f\/irst homology group of $V/\Pi$ as a~right $P$-module:
\begin{gather*}
\Pi^* = \operatorname{Hom}(\Pi, \Z) \cong \operatorname{Hom}(H_1(V/\Pi; \Z), \Z) \cong H^1(V/\Pi; \Z).
\end{gather*}
Considering the dual space $V^* = \operatorname{Hom}(V, \R)$ and its lattice $\Pi^*$ instead, we similarly get an isomorphism of left $P$-modules
\begin{gather*}
\Pi \cong H^1(V^*/\Pi^*; \Z).
\end{gather*}
Since there is a natural isomorphism of tori $V^*/\Pi^* \to \hat{\Pi} = \operatorname{Hom}(\Pi, U(1))$ with right $P$-actions, the isomorphism of the group cohomologies is proved. The factoring of the isomorphism can be verif\/ied by a direct inspection.
\end{proof}

Now, in the case of \textsf{pm/pg}, the nonsymmorphic group \textsf{pg} def\/ines the non-trivial element of $H^2_{\mathrm{group}}(\Z_2; \Pi) \cong \Z_2$ through $\nu$, and the element corresponds by Lemma~\ref{lem:cocycle_correspondence} to the non-trivial element of $H^2_{\mathrm{group}}(\Z_2; H^1(T^2; \Z)) \cong \Z_2$ represented by the group $2$-cocycle $\tau$ induced from \textsf{pg}. The same holds true in the case of \textsf{p4m/p4g}. In the case of \textsf{pmm/pmg/pgg}, we have $H^2_{\mathrm{group}}(D_2; \Pi) \cong H^2_{\mathrm{group}}(D_2; H^1(T^2; \Z)) \cong \Z_2 \oplus \Z_2$. In view of the classif\/ication of $2$-dimensional space groups (\cite{H}), the non-trivial elements $(-1, 1)$ and $(-1, -1)$, with respect to a basis of $\Z_2 \oplus \Z_2$, correspond to the nonsymmorphic groups \textsf{pmg} and \textsf{pgg} respectively. (The nonsymmorphic group corresponding to $(1, -1)$ is equivalent to \textsf{pmg} through an af\/f\/ine transformation preserving the lattice.) Therefore $H^2_{\mathrm{group}}(D_2; H^1(T^2; \Z)) \cong \Z_2 \oplus \Z_2$ is generated by the group $2$-cocycles induced from the nonsymmorphic groups.

\section{The twisted case}\label{sec:twisted_case}

This section concerns the equivariant cohomology with local coef\/f\/icients. We start with some remarks about the Leray--Serre spectral sequence, focusing on the similarities and the dif\/ferences with the case of the usual Borel equivariant cohomology. We then summarize tools for computation in the version adapted to the case with local coef\/f\/icients. After that, we prove Theo\-rems~\ref{thm:main_twisted} and~\ref{thm:equiv_coh_twisted}: As in the untwisted case, the full computation is lengthy, and the details are only provided in the case of~\textsf{p6m}.

\subsection{The Leray--Serre spectral sequence}

For a f\/inite group $G$ and a homomorphism $\phi \colon G \to \Z_2$, we def\/ine $\Z_\phi$ to be the $G$-module $\Z_\phi$ such that its underlying group is $\Z$ and $G$ acts via $\phi$. Suppose that $G$ acts on a space $X$. Associated to the f\/ibration $X \to EG \times_G X \to BG$ is the long exact sequence of homotopy groups:
\begin{gather*}
\cdots \to \pi_n(X) \to \pi_n(EG \times_G X) \to \pi_n(BG) \to \cdots.
\end{gather*}
Since $\pi_1(BG) \cong G$, we get a homomorphism $\pi_1(EG \times_G X) {\to} \Z_2$ by composing \mbox{$\pi_1(EG \times_G X) {\to} G$} with $\phi \colon G \to \Z_2$. The homomorphism def\/ines a local system on the Borel construction $EG \times_G X$, which we denote with the same notation $\Z_\phi$. By using this local system, we def\/ine the $G$-equivariant cohomology with local coef\/f\/icients
\begin{gather*}
H^n_G(X; \Z_\phi) = H^n(EG \times_G X; \Z_\phi).
\end{gather*}
Of course, if $\phi$ is trivial, then the cohomology above recovers the usual equivariant cohomology with integer coef\/f\/icients $\Z$.

For $H^n_G(X; \Z_\phi)$, we also have the Leray--Serre spectral sequence $E_r^{p, q}$ converging to the graded quotient of a f\/iltration
\begin{gather*}
H^n_G(X; \Z_\phi) = F^0H^n_G(X; \Z_\phi) \supset
F^1H^n_G(X; \Z_\phi) \supset F^2H^n_G(X; \Z_\phi) \supset \cdots.
\end{gather*}
Its $E_2$-term is again given by the group cohomology
\begin{gather*}
E_2^{p, q} = H^p_{\mathrm{group}}(G; H^q(X; \Z) \otimes \Z_\phi),
\end{gather*}
where $H^q(X; \Z) \otimes \Z_\phi$ is the tensor product of the $G$-modules $H^q(X; \Z)$ and $\Z_\phi$, namely, its underlying group is
\begin{gather*}
H^q(X; \Z) \otimes \Z_\phi = H^q(X; \Z) \otimes \Z
\cong H^q(X; \Z),
\end{gather*}
and $g \in G$ acts on $x \in H^q(X; \Z)$ by $x \mapsto \phi(g) \cdot g^*x$.

As in the usual equivariant cohomology, $H^n_G(X; \Z_\phi)$ can be identif\/ied with a sheaf cohomology of the simplicial space $G^\bullet \times X$. (See for instance the appendix of \cite{G} in the case of $G = \Z_2$.) This interpretation leads to the classif\/ication of the twists for the Freed--Moore $K$-theory \cite{F-M}, whose def\/inition is similar to the one given in Section \ref{subsec:twist} (see \cite{Ku}). In terms of the Borel equivariant cohomology, the graded twists are classif\/ied by $H^1_G(X; \Z_2) \times H^3_G(X; \Z_\phi)$ and the ungraded twists by $H^3_G(X; \Z_\phi)$.

The results in Section \ref{subsec:comparison} can be generalized to the equivariant cohomology with coef\/f\/icients in $\Z_\phi$, and we get the following geometric interpretation generalizing Corollary \ref{cor:twist_interpretation}:

\begin{prop} \label{prop:phi_twist_interpretation}
Let $G$ be a finite group acting on a compact and path connected space $X$ with at least one fixed point, and $\phi \colon G \to \Z_2$ a homomorphism.
\begin{itemize}\itemsep=0pt
\item[$(i)$]
$F^1H^3_G(X; \Z_\phi)$ classifies $($ungraded$)$ twists which can be represented by $\phi$-twisted central extensions of the groupoid $X/\!/G$.

\item[$(ii)$]
$F^2H^3_G(X; \Z_\phi)$ classifies $($ungraded$)$ twists which can be represented by $2$-cocycles of $G$ with coefficients in the $G$-module $C(X, U(1))_\phi$, where $C(X, U(1))_\phi = C(X, U(1))$ is the group of $U(1)$-valued functions on $X$ and $g \in G$ acts on $f \colon X \to U(1)$ by $f \mapsto g^*f^{\phi(g)}$.

\item[$(iii)$]
$F^3H^3_G(X; \Z_\phi) = H^2_{\mathrm{group}}(G; U(1)_\phi)$ classifies $($ungraded$)$ twists which can be represented by $2$-cocycles of $G$ with coefficients in the $G$-module $U(1)_\phi$, where $U(1) = U(1)_\phi$ as a group and $g \in G$ acts on $u \in U(1)$ by $u \mapsto u^{\phi(g)}$.
\end{itemize}
\end{prop}

\subsection{Tools}

As long as we are concerned with the local system $\Z_\phi$ associated to a homomorphism $\phi \colon G \to \Z_2$, the reduced cohomology $\tilde{H}^n_G(X; \Z_\phi)$ makes sense for a $G$-space $X$ with a f\/ixed point $\pt \in X$, and we have the direct sum decomposition $H^n_G(X; \Z_\phi) \cong H^n_G(\pt; \Z_\phi) \oplus \tilde{H}^n_G(X; \Z_\phi)$. Then we can generalize the proof of Lemma~\ref{lem:1_dim_complex} to show:

\begin{lem}
Suppose that a finite group $G$ acts on a path connected space $Y$ fixing at least one point $\pt \in Y$. Suppose further that $Y$ is a CW complex consisting of only cells of dimension less than or equal to~$1$. Then, for any homomorphism $\phi \colon G \to \Z_2$ and $n \ge 0$, the following holds true:
\begin{gather*}
H^n_{\mathrm{group}}\big(G; H^1(Y; \Z) \otimes \Z_\phi\big) \cong \tilde{H}_G^{n+1}(Y; \Z_\phi).
\end{gather*}
\end{lem}

Similarly, we can generalize Lemma \ref{lem:equiv_coh_splitting_case} as follows:

\begin{lem} \label{lem:twisted_equiv_coh_splitting_case}
Suppose a finite group $G$ acts on the torus $T^2 = S^1 \times S^1$ and
\begin{itemize}\itemsep=0pt
\item there is a fixed point $\pt = (x_0, y_0) \in T^2$ under the $G$-action,

\item $G$ preserves the subspace $S^1 \vee S^1 = S^1 \times \{ y_0 \} \cup \{ x_0 \} \times S^1 \subset T^2$.
\end{itemize}
Then there is the following isomorphism of groups for any homomorphism $\phi \colon G \to \Z_2$ and all $n \in \Z$
\begin{gather*}
H^n_G\big(T^2; \Z_\phi\big) \cong H^n_G(\pt; \Z_\phi) \oplus \tilde{H}^n_G\big(S^1 \vee S^1; \Z_\phi\big) \oplus \tilde{H}^n_G\big(T^2/S^1 \vee S^1; \Z_\phi\big).
\end{gather*}
Further, the Leray--Serre spectral sequence for $H^n_G(T^2; \Z_\phi)$ degenerates at $E_2$ and the relevant extension problems are trivial, so that
\begin{itemize}\itemsep=0pt
\item[$(a)$]
$F^2H^3_G(T^2; \Z_\phi) \cong E_2^{2, 1} \oplus E_2^{3, 0}$,

\item[$(b)$]
$H^n_G(T^2; \Z_\phi) \cong \bigoplus_{p + q = n} E_2^{p, q}$ for all $n \in \Z$.
\end{itemize}
\end{lem}

Besides the generalizations above, we will use the following universal coef\/f\/icient theorem in the sequel:

\begin{lem} \label{lem:Z2_coefficient}
For any $\phi \colon G \to \Z_2$, there is a split exact sequence of groups
\begin{gather*}
0 \to H^n_G(X; \Z_\phi) \otimes \Z_2 \to H^n_G(X; \Z_2) \to \operatorname{Tor}\big(H^{n+1}_G(X; \Z_\phi), \Z_2\big) \to 0.
\end{gather*}
\end{lem}

\begin{proof}
For any homomorphism $\phi \colon G \to \Z_2$, let $(\Z_2)_\phi$ be the $G$-module such that its underlying group is $\Z_2$ and its $G$-action is given by $\phi \colon G \to \Z_2$. This $G$-module $(\Z_2)_\phi$ agrees with the trivial $G$-module $\Z_2$, even if $\phi$ is non-trivial. Then, looking at the cochain complexes def\/ining the equivariant cohomology, the usual proof of the universal coef\/f\/icient theorem leads to the lemma. Another proof is to use the Thom isomorphism, which unwinds the local coef\/f\/icients: Let $\underline{\R}_{\phi} \to X$ be the $G$-equivariant real line bundle on $X$ whose underlying bundle is $X \times \R$ and the action of $g \in G$ on $(x, r) \in X \times \R$ is $(x, r) \mapsto (gx, \phi(g)r)$. The Thom isomorphism theorem then provides
\begin{gather*}
H^n_G(X; A_\phi) \cong H^{n+1}_G(D, S; A),
\end{gather*}
where $A$ is $\Z_2$ or $\Z$, and $D \subset \underline{\R}_\phi$ and $S \subset \underline{\R}_\phi$ are the unit disk bundle and the unit sphere bundle, respectively. Then the usual universal coef\/f\/icient theorem leads to the present lemma.
\end{proof}

To compute the equivariant cohomology $H^n_G(X; \Z_\phi)$, we usually need the cohomology of the point $H^n_G(\pt; \Z_\phi)$. This cohomology is identif\/ied with the group cohomology $H^n_{\mathrm{group}}(G; \Z_\phi)$ by the degeneration of the Leray--Serre spectral sequence, but its direct computation is not realistic except for the simplest cases (cf.\ Lemma~\ref{lem:group_coh_orientation_coeff}). A useful way to compute it is:

\begin{lem} \label{lem:exact_seq}
For any $\phi \colon G \to \Z_2$, there are natural exact sequences
\begin{gather*}
\cdots \to H^{n-1}_G(\pt; \Z_\phi) \to H^n_G(\pt; \Z) \overset{i^*}{\to} H^n_{\operatorname{Ker}\phi}(\pt; \Z) \to H^n_G(\pt; \Z_\phi) \to \cdots, \\
\cdots \to H^{n-1}_G(\pt; \Z) \to H^n_G(\pt; \Z_\phi) \overset{i^*}{\to} H^n_{\operatorname{Ker}\phi}(\pt; \Z) \to H^n_G(\pt; \Z) \to \cdots,
\end{gather*}
where $i^*$ is induced from the inclusion $i \colon \operatorname{Ker}\phi \to G$.
\end{lem}

\begin{proof}
Let $G$ act on $\R_\phi = \R$ via $\phi \colon G \to \Z_2$. We can regard $\R_\phi$ as a $G$-equivariant real line bundle on $\pt$. We have the Thom isomorphisms
\begin{gather*}
H^n_G(\pt; \Z_\phi) \cong H^{n+1}_G(D, S; \Z), \qquad H^n_G(\pt; \Z) \cong H^{n+1}_G(D, S; \Z_\phi),
\end{gather*}
where $D$ and $S$ are the unit interval in $\R_\phi$ and its boundary, respectively. Note that $D$ is equivariantly contractible. Note also that $S \cong G/\operatorname{Ker}\phi$ as a $G$-space. Thus, we have isomorphisms
\begin{gather*}
H^n_G(D; \Z_\phi) \cong H^n_G(\pt; \Z_\phi), \qquad H^n_G(S; \Z_\phi) \cong H^n_{\operatorname{Ker}\phi}(\pt; \Z).
\end{gather*}
Substituting these isomorphisms and the Thom isomorphisms into the exact sequences for the pair $(D, S)$, we complete the proof.
\end{proof}

By means of the lemma above, we get:

\begin{lem} \label{lem:cohomology_of_point_twisted_case}
Let $P$ be $\Z_{2m}$, $D_{2m-1}$ or $D_{2m}$ with $m \ge 1$. The $P$-equivariant cohomology of the point with coefficients in $\Z_\phi$,
\begin{gather*}
H^n_P(\pt; \Z_\phi) \cong H^n_{\mathrm{group}}(P; \Z_\phi),
\end{gather*}
in low degrees is given as follows:
$$
\begin{array}{|c|c|c|c|c|c|}
\hline
P & \phi &
H^0_P(\pt; \Z_\phi) & H^1_P(\pt; \Z_\phi) &
H^2_P(\pt; \Z_\phi) & H^3_P(\pt; \Z_\phi) \tsep{2pt}\bsep{2pt}\\
\hline
\Z_{2m} & \phi_1 &
0 & \Z_2 & 0 & \Z_2 \\
\hline
D_{2m-1} & \phi_0 &
0 & \Z_2 & \Z_{2m-1} & \Z_2 \\
\hline
D_{2m} & \phi_0 &
0 & \Z_2 & \Z_{2m} & \Z_2^{\oplus 2} \tsep{2pt}\bsep{2pt}\\
\hline
D_{2m} & \phi_1, \phi_2 &
0 & \Z_2 & \Z_2 & \Z_2^{\oplus 2} \tsep{2pt}\bsep{2pt}\\
\hline
\end{array}
$$
\end{lem}

\begin{proof}
In the case of $D_{2m}$ with $m$ even and $\phi \neq \phi_0$, the f\/irst exact sequence in Lemma~\ref{lem:exact_seq} leads to $H^0_{D_{2m}}(\pt; \Z_\phi) = 0$ and $H^1_{D_{2m}}(\pt; \Z_\phi) \cong \Z_2$. This computation also shows that $H^2_{D_{2m}}(\pt; \Z_\phi)$ contains $\Z_2$ as a subgroup. Here, applying the universal coef\/f\/icient theorem to $H^n_P(\pt; \Z)$, we compute the cohomology with coef\/f\/icients in $\Z_2$ to have $H^1_{D_{2m}}(\pt; \Z_2) \cong \Z_2^{\oplus 2}$. If we apply the universal coef\/f\/icient theorem in Lemma \ref{lem:Z2_coefficient} to $H^n_P(\pt; \Z_\phi)$, then
\begin{gather*}
H^1_{D_{2m}}(\pt; \Z_2) \cong \Z_2 \oplus \operatorname{Tor}\big(H^2_{D_{2m}}(\pt; \Z_\phi), \Z_2\big).
\end{gather*}
Thus the consistency of these computations implies $H^2_{D_{2m}}(\pt; \Z_\phi) \cong \Z_2$. Based on this result, the second sequence in Lemma~\ref{lem:exact_seq} suggests that $H^3_{D_{2m}}(\pt; \Z_\phi)$ is either $\Z_2^{\oplus 2}$ or $\Z_4$. If we apply the universal coef\/f\/icient theorem to $H^n_P(\pt; \Z)$, then $H^2_{D_{2m}}(\pt; \Z_2) \cong \Z_2^{\oplus 3}$. If we compute this cohomology applying Lemma~\ref{lem:Z2_coefficient} to $H^n_P(\pt; \Z_\phi)$, then
\begin{gather*}
H^2_{D_{2m}}(\pt; \Z_2) \cong \Z_2 \oplus \operatorname{Tor}\big(H^3_{D_{2m}}(\pt; \Z_\phi), \Z_2\big).
\end{gather*}
Therefore we conclude that $H^3_{D_{2m}}(\pt; \Z_\phi), \Z_2) \cong \Z_2^{\oplus 2}$ by the consistency. In the other cases, a~combined use of the two exact sequences in Lemma~\ref{lem:exact_seq} determines the group $H^n_P(\pt; \Z_\phi)$ without dif\/f\/iculty.
\end{proof}

\subsection{The proof of Theorems \ref{thm:main_twisted} and \ref{thm:equiv_coh_twisted}}

Theorems~\ref{thm:main_twisted} and~\ref{thm:equiv_coh_twisted} again follow from case-by-case computations. To these cases, we can apply the methods in the proof of Theorems~\ref{thm:main} and~\ref{thm:equiv_coh}. However, in some cases, only the possibility of a cohomology group is suggested by an exact sequence. In this case, we apply an argument used in the proof of Lemma~\ref{lem:cohomology_of_point_twisted_case}: We compute the cohomology with coef\/f\/icients in $\Z_2$ applying the universal coef\/f\/icient theorem to the result in Theorem~\ref{thm:equiv_coh}. Then the consistency with Lemma~\ref{lem:Z2_coefficient} eventually determines the cohomology in question.

In the following, we carry out the computation in the case of \textsf{p6m} with $\phi = \phi_2$. Let $Y \subset T^2$ be the $D_6$-invariant subspace given in Section~\ref{subsec:p6m}.

\begin{lem}
The $D_6$-equivariant cohomology of $Y$ with coefficients in $\Z_{\phi_2}$ in low degrees is as follows:
$$
\begin{array}{|c|c|c|c|c|}
\hline
 & n = 0 & n = 1 & n = 2 & n = 3 \\
\hline
H^n_{D_6}(Y; \Z_{\phi_2}) & 0 & \Z_2 & \Z_2^{\oplus 2} & \Z_2^{\oplus 3} \tsep{2pt}\bsep{2pt}\\
\hline
\end{array}
$$
\end{lem}

\begin{proof}
To use the Mayer--Vietoris sequence, we cover $Y$ by $D_6$-invariant subspaces $U$ and $V$ which have the following $D_6$-equivariant homotopy equivalences
\begin{gather*}
U \simeq \pt, \qquad V \simeq D_6/D_2, \qquad U \cap V \simeq D_6/\Z_2,
\end{gather*}
where $D_2 = \{ 1, C^3, \sigma_1, \sigma_4 \} \subset D_6$ and $\Z_2 = \{ 1, \sigma_1 \} \subset D_6$. We see
\begin{gather*}
H^n_{D_6}(V; \Z_{\phi_2}) \cong H^n_{D_2}(\pt; \Z_{\phi_2}), \qquad
H^n_{D_6}(U \cap V; \Z_{\phi_2}) \cong H^n_{\Z_2}(\pt; \Z_{\phi_1}).
\end{gather*}
The equivariant cohomology groups in low degrees can be summarized as follows:
$$
\begin{array}{|c|c|c|c|}
\hline
n = 3 & & \Z_2^{\oplus 2} \oplus \Z_2^{\oplus 2} & \Z_2 \tsep{2pt}\bsep{2pt}\\
\hline
n = 2 & & \Z_2 \oplus \Z_2 & 0 \\
\hline
n = 1 & & \Z_2 \oplus \Z_2 & \Z_2 \\
\hline
n = 0 & & 0 \oplus 0 & 0 \\
\hline
& H^n_{D_6}(Y; \Z_{\phi_2}) & H^n_{D_6}(U \sqcup V; \Z_{\phi_2}) &
H^n_{D_6}(U \cap V; \Z_{\phi_2})\tsep{2pt}\bsep{2pt}\\
\hline
\end{array}
$$
We have $H^0_{D_6}(Y; \Z_{\phi_2}) = 0$ clearly, and $H^1_{D_6}(Y; \Z_{\phi_2})$ is either $\Z_2$ or $\Z_2^{\oplus 2}$. Applying the universal coef\/f\/icient theorem to Lemma \ref{lem:p6m_Y}, we f\/ind $H^0_{D_6}(Y; \Z_2) \cong \Z_2$. This result must be consistent with the computation of $H^0_{D_6}(Y; \Z_2)$ by using Lemma \ref{lem:Z2_coefficient}, which leads to $H^1_{D_6}(Y; \Z_{\phi_2}) \cong \Z_2$. Solving the Mayer--Vietoris exact sequence, we then f\/ind $H^2_{D_6}(Y; \Z_{\phi_2}) \cong \Z_2^{\oplus 2}$. We also f\/ind that $H^3_{D_6}(Y; \Z_{\phi_2})$ is either $\Z_2^{\oplus 3}$ or $\Z_2^{\oplus 4}$. Computing again the cohomology with $\Z_2$-coef\/f\/icients in two ways, we conclude that $H^3_{D_6}(Y; \Z_{\phi_2}) \cong \Z_2^{\oplus 3}$.
\end{proof}

\begin{lem}
There is an exact sequence of $D_6$-modules
\begin{gather*}
0 \to H^1\big(T^2; \Z\big) \otimes \Z_{\phi_2} \to H^1(Y; \Z) \otimes \Z_{\phi_2} \overset{\pi}{\to}\Z_{\phi_0} \to 0
\end{gather*}
admitting a module homomorphism $s \colon \Z_{\phi_0} \to H^1(Y; \Z) \otimes \Z_{\phi_2}$ such that $\pi \circ s = 3$.
\end{lem}

\begin{proof}
The proof of Lemma~\ref{lem:p6m_short_exact_sequence_of_modules} can be adapted to this case.
\end{proof}

\begin{lem}
$H^n_{\mathrm{group}}(D_6; H^1(T^2; \Z) \otimes \Z_{\phi_2}) = 0$ for $n = 0, 1, 2$.
\end{lem}

\begin{proof}
We use the long exact sequence of group cohomology induced from the short exact sequence of coef\/f\/icients. Notice that
\begin{gather*}
H^n_{\mathrm{group}}\big(D_6; H^1(Y; \Z) \otimes \Z_{\phi_2}\big) \cong \tilde{H}^{n+1}_{D_6}(Y; \Z_{\phi_2}).
\end{gather*}
The relevant cohomology can be summarized as follows:
$$
\begin{array}{|c|c|c|c|}
\hline
2 & & \Z_2 & \Z_6 \\
\hline
1 & & \Z_2 & \Z_2 \\
\hline
0 & & 0 & 0 \\
\hline
n & H^n_{\mathrm{group}}(D_6; H^1(T^2) \otimes \Z_{\phi_2}) &
H^n_{\mathrm{group}}(D_6; H^1(Y) \otimes \Z_{\phi_2}) &
H^n_{\mathrm{group}}(D_6; \Z_{\phi_0})\tsep{2pt}\bsep{2pt}\\
\hline
\end{array}
$$
By $s \colon \Z_{\phi_0} \to H^1(Y; \Z) \otimes \Z_{\phi_2}$, the group cohomology is determined as stated.
\end{proof}

\begin{thm}[\textsf{p6m} with $\phi_2$]
The $D_6$-equivariant cohomology of $T^2$ with coefficients in $\Z_{\phi_2}$ in low degrees is given as follows:
$$
\begin{array}{|c|c|c|c|c|}
\hline
& n = 0 & n = 1 & n = 2 & n = 3 \\
\hline
H^n_{D_6}(T^2; \Z_{\phi_2}) & 0 & \Z_2 & \Z_2 & \Z_2^{\oplus 3} \tsep{2pt}\bsep{2pt}\\
\hline
\end{array}
$$
We also have: $F^2H^3_{D_6}(T^2; \Z_{\phi_2}) \cong F^3H^3_{D_6}(T^2; \Z_{\phi_2}) \cong \Z_2^{\oplus 2}$.
\end{thm}

\begin{proof}
In the $E_2$-term of the Leray--Serre spectral sequence, we have the following identif\/ications
\begin{gather*}
H^n_{\mathrm{group}}\big(D_6; H^0\big(T^2; \Z\big) \otimes \Z_{\phi_2}\big) \cong H^n_{D_6}(\pt; \Z_{\phi_2}), \\
H^n_{\mathrm{group}}\big(D_6; H^2\big(T^2; \Z\big) \otimes \Z_{\phi_2}\big)\cong H^n_{D_6}(\pt; \Z_{\phi_1}).
\end{gather*}
We can summarize the $E_2$-terms as follows:
$$
\begin{array}{|c|c|c|c|c|}
\hline
q = 3 & 0 & 0 & 0 & 0 \\
\hline
q = 2 & 0 & \Z_2 & & \\
\hline
q = 1 & 0 & 0 & 0 & \\
\hline
q = 0 & 0 & \Z_2 & \Z_2 & \Z_2^{\oplus 2} \tsep{2pt}\bsep{2pt} \\
\hline
E_2^{p, q} & p = 0 & p = 1 & p = 2 & p = 3 \tsep{2pt}\bsep{2pt}\\
\hline
\end{array}
$$
Because $E_2^{n, 0}$ must survive into the direct summand $H^n_{D_6}(\pt; \Z_{\phi_2})$ of the cohomology \linebreak $H^n_{D_6}(T^2; \Z_{\phi_2})$, we get the degeneration $E_2^{p, q} = E_\infty^{p, q}$ for $p + q \le 2$, and the relevant extension problems are readily solved. We also have $H^3_{D_6}(T^2; \Z_{\phi_2}) \cong \Z_2^{\oplus 2} \oplus E_\infty^{1, 2}$, where $E_\infty^{1, 2} \subset E_2^{1, 2} = \Z_2$ is either~$\Z_2$ or~$0$. By computing the cohomology with $\Z_2$-coef\/f\/icients in two ways, we conclude that $E_\infty^{1, 2} = E_2^{1, 2} = \Z_2$.
\end{proof}

\appendix

\section{The list of 2-dimensional space groups}\label{sec:appendix:wallpapers}

Here is a list of the lattices $\Pi$ and the point groups $P$ of the $2$-dimensional space groups $S$. In the nonsymmorphic case, the map $a\colon P \to \R^2$ in Section~\ref{sec:quantum_system_to_K} is also presented.

\subsection{Oblique, rectangular and square lattices}

For \textsf{p1}, \textsf{p2}, \textsf{p4}, \textsf{pm}, \textsf{pg}, \textsf{pmm}, \textsf{pmg}, \textsf{pgg}, \textsf{p4m} and \textsf{p4g}, we can take the lattice $\Pi \subset \R^2$ to be the standard lattice $\Pi = \Z^2$.
\begin{itemize}\itemsep=0pt
\item
(\textsf{p1})
The point group is trivial.

\item
(\textsf{p2})
The point group $\Z_2 = \langle C \,|\, C^2 \rangle$ acts on $\Pi$ and $\R^2$ through the matrix
\begin{gather*}
C =
\left(
\begin{matrix}
-1 & 0 \\
0 & -1
\end{matrix}
\right).
\end{gather*}

\item
(\textsf{p4})
The point group $\Z_4 = \langle C \,|\, C^4 \rangle$ acts on $\Pi$ and $\R^2$ through
\begin{gather*}
C =
\left(
\begin{matrix}
0 & -1 \\
1 & 0
\end{matrix}
\right).
\end{gather*}

\item
(\textsf{pm/pg})
The point group $D_1 = \langle \sigma \,|\, \sigma^2 \rangle$ acts on $\Pi$ and $\R^2$ through
\begin{gather*}
\left(
\begin{matrix}
-1 & 0 \\
0 & 1
\end{matrix}
\right).
\end{gather*}
In the case of \textsf{pg}, the map $a \colon D_1 \to \R^2$ is given by
\begin{gather*}
a_1 =
\left(
\begin{matrix}
0 \\ 0
\end{matrix}
\right),\qquad
a_{\sigma} =
\left(
\begin{matrix}
0 \\ 1/2
\end{matrix}
\right).
\end{gather*}

\item
(\textsf{pmm/pmg/pgg})
The point group is $D_2 \cong \Z_2 \times \Z_2$. We let the following matrices $\sigma_x$ and $\sigma_y$ generate $D_2$, and act on $\Pi$ and $\R^2$
\begin{gather*}
\sigma_x =
\left(
\begin{matrix}
-1 & 0\\
0 & 1
\end{matrix}
\right), \qquad
\sigma_y =
\left(
\begin{matrix}
1 & 0\\
0 & -1
\end{matrix}
\right).
\end{gather*}
In the case of \textsf{pmg}, the map $a\colon D_2 \to \R^2$ is given by
\begin{gather*}
a_1 =
\left(
\begin{matrix}
0 \\ 0
\end{matrix}
\right),\qquad
a_{\sigma_x} =
\left(
\begin{matrix}
0 \\ 1/2
\end{matrix}
\right),
\qquad
a_{\sigma_y} =
\left(
\begin{matrix}
0 \\ 0
\end{matrix}
\right),
\qquad
a_{\sigma_x\sigma_y} =
\left(
\begin{matrix}
0 \\ 1/2
\end{matrix}
\right).
\end{gather*}
In the case of \textsf{pgg}, the map $a \colon D_2 \to \R^2$ is given by
\begin{gather*}
a_1 =
\left(
\begin{matrix}
0 \\ 0
\end{matrix}
\right),
\qquad
a_{\sigma_x} =
\left(
\begin{matrix}
0 \\ 1/2
\end{matrix}
\right),
\qquad
a_{\sigma_y} =
\left(
\begin{matrix}
1/2 \\ 0
\end{matrix}
\right),
\qquad
a_{\sigma_x\sigma_y} =
\left(
\begin{matrix}
1/2 \\ 1/2
\end{matrix}
\right).
\end{gather*}

\item
(\textsf{p4m/p4g})
The point group is $D_4 = \langle C_4, \sigma_x \,|\, C_4^4, \sigma_x^2, \sigma_xC_4\sigma_xC_4 \rangle$, which acts on $\Pi$ and $\R^2$ through the following matrix presentation
\begin{gather*}
C_4 =
\left(
\begin{matrix}
0 & -1\\
1 & 0
\end{matrix}
\right), \qquad
\sigma_x =
\left(
\begin{matrix}
-1 & 0\\
0 & 1
\end{matrix}
\right).
\end{gather*}
In the case of \textsf{p4g}, the map $a \colon D_4 \to \R^2$ is as follows:
$$
\begin{array}{|c|cccc|cccc|}
\hline
p &
1 & C_4 & C_4^2 & C_4^3 & \sigma_x & \sigma_d & \sigma_y & \sigma_d'\tsep{2pt}\bsep{2pt}\\
\hline
a_p &
\tsep{10pt}\left[
\begin{matrix}
0 \\ 0
\end{matrix}
\right]
&
\left[
\begin{matrix}
0 \\ \frac{1}{2}
\end{matrix}
\right]
&
\left[
\begin{matrix}
\frac{1}{2} \\ \frac{1}{2}
\end{matrix}
\right]
&
\left[
\begin{matrix}
\frac{1}{2} \\ 0
\end{matrix}
\right]
&
\left[
\begin{matrix}
0 \\ \frac{1}{2}
\end{matrix}
\right]
&
\left[
\begin{matrix}
0 \\ 0
\end{matrix}
\right]
&
\left[
\begin{matrix}
\frac{1}{2} \\ 0
\end{matrix}
\right]
&
\left[
\begin{matrix}
\frac{1}{2} \vspace{1mm}\\ \frac{1}{2}
\end{matrix}
\right] \bsep{10pt}\\
\hline
\end{array}
$$
In the above, we set $\sigma_d = \sigma_xC_4$, $\sigma_y = C_4^2\sigma_x$ and $\sigma_d' = C_4\sigma_x$.
\end{itemize}

\subsection{Rhombic lattices}

For \textsf{cm} and \textsf{cmm}, the lattice is $\Pi = \Z a \oplus \Z b \subset \R^2$, where
\begin{gather*}
a = \left(\begin{matrix} 1 \\ 1 \end{matrix} \right), \qquad b = \left(\begin{matrix} -1 \\ 1 \end{matrix} \right).
\end{gather*}
\begin{itemize}\itemsep=0pt
\item
(\textsf{cm})
The point group $D_1 = \langle \sigma | \sigma^2 \rangle$ acts on $\Pi$ and $\R^2$ by
\begin{gather*}
\sigma =
\left(
\begin{matrix}
-1 & 0 \\
0 & 1
\end{matrix}
\right).
\end{gather*}

\item
(\textsf{cmm})
The point group is $D_2 \cong \Z_2 \times \Z_2$. The following matrices $\sigma_x$ and $\sigma_y$ generate $D_2$, and def\/ine the $D_4$-action on $\Pi$ and $\R^2$
\begin{gather*}
\sigma_x = \left(\begin{matrix}
-1 & 0 \\
0 & 1
\end{matrix}\right),
\qquad
\sigma_y =\left(\begin{matrix}
1 & 0 \\
0 & -1
\end{matrix}
\right).
\end{gather*}
\end{itemize}

\subsection{Hexagonal lattices}

For \textsf{p3}, \textsf{p6}, \textsf{p3m1}, \textsf{p31m} and \textsf{p6m}, the lattice $\Pi = \Z a \oplus \Z b \subset \R^2$ is spanned by
\begin{gather*}
a = \left( \begin{matrix} 1 \\ 0 \end{matrix} \right), \qquad
b = \left( \begin{matrix} 1/2 \\ \sqrt{3}/2 \end{matrix} \right).
\end{gather*}

\begin{itemize}\itemsep=0pt
\item
(\textsf{p3})
The point group $\Z_3 = \langle C \,|\, C^3 \rangle$ acts on $\Pi$ and $\R^2$ through
\begin{gather*}
C =
\left(
\begin{matrix}
-1/2 & -\sqrt{3}/2 \\
\sqrt{3}/2 & -1/2
\end{matrix}
\right).
\end{gather*}

\item
(\textsf{p6})
The point group $\Z_6 = \langle C \,|\, C^6 \rangle$ acts on $\Pi$ and $\R^2$ through
\begin{gather*}
C =
\left(
\begin{matrix}
1/2 & -\sqrt{3}/2 \\
\sqrt{3}/2 & 1/2
\end{matrix}
\right).
\end{gather*}

\item
(\textsf{p3m1})
The point group is $D_3 = \langle C, \sigma_x \,|\, C^3, \sigma_x^2, \sigma_xC\sigma_xC \rangle$. We let $D_3$ act on $\Pi$ and $\R^2$ through the inclusion $D_3 \subset {\rm O}(2)$ given by
\begin{gather*}
C =
\left(
\begin{matrix}
-1/2 & -\sqrt{3}/2 \\
\sqrt{3}/2 & -1/2
\end{matrix}
\right), \qquad
\sigma_x =
\left(
\begin{matrix}
-1 & 0 \\
0 & 1
\end{matrix}
\right).
\end{gather*}

\item
(\textsf{p31m})
The point group is $D_3 = \langle C, \sigma_y \,|\, C^3, \sigma_y^2, \sigma_yC\sigma_yC \rangle$. We let $D_3$ act on $\Pi$ and $\R^2$ through the inclusion $D_3 \subset {\rm O}(2)$ given by
\begin{gather*}
C =
\left(
\begin{matrix}
-1/2 & -\sqrt{3}/2 \\
\sqrt{3}/2 & -1/2
\end{matrix}
\right), \qquad
\sigma_y =
\left(
\begin{matrix}
1 & 0 \\
0 & -1
\end{matrix}
\right).
\end{gather*}

\item
(\textsf{p6m})
The point group is $D_6 = \langle C, \sigma_1 \,|\, C^6, \sigma_1^2, \sigma_1C\sigma_1C \rangle$. We let $D_6$ act on $\Pi$ and $\R^2$ through the inclusion $D_6 \subset {\rm O}(2)$ given by
\begin{gather*}
C =
\left(
\begin{matrix}
1/2 & -\sqrt{3}/2 \\
\sqrt{3}/2 & 1/2
\end{matrix}
\right), \qquad
\sigma_1 =
\left(
\begin{matrix}
1 & 0 \\
0 & -1
\end{matrix}
\right).
\end{gather*}
\end{itemize}

\subsection*{Acknowledgements}

I would like to thank K.~Shiozaki and M.~Sato for valuable discussions. I would also thank G.C.~Thiang, D.~Tamaki, anonymous referees and an editor for helpful criticisms and comments. This work is supported by JSPS KAKENHI Grant Number JP15K04871.

\pdfbookmark[1]{References}{ref}
\LastPageEnding

\end{document}